\documentclass[11pt]{amsart}

\usepackage{amsfonts, amssymb, amscd}

\usepackage{bm}
\usepackage{verbatim}
\usepackage{amssymb}
\usepackage{mathrsfs}
\usepackage{graphicx}
\usepackage[nospace,noadjust]{cite}
\usepackage[all]{xy}
\usepackage{tikz}
\usepackage{subcaption}
\usepackage{subfiles}
\usepackage[toc,page]{appendix}
\usepackage{mathtools}
\usepackage{comment}
\usepackage{enumerate}
\usepackage{enumitem}

\usepackage{graphicx}
\graphicspath{ {images/} }

\usepackage{appendix}
\usepackage{hyperref}

\newcommand{\Zz}{\mathbb{Z}}

\newcommand{\Pp}{\mathbb{P}}
\newcommand{\Rr}{\mathbb{R}}
\newcommand{\Qq}{\mathbb{Q}}
\newcommand{\Nn}{\mathbb{N}}

\newcommand{\Exc}{\operatorname{Exc}}
\newcommand{\Supp}{\operatorname{Supp}}
\newcommand{\Center}{\operatorname{Center}}

\newcommand{\lf}{\lfloor}
\newcommand{\rf}{\rfloor}

\newcommand{\Bb}{\mathcal{B}}

\newcommand{\Oo}{\mathcal{O}}
\newcommand{\Ii}{\mathcal{I}}

\newcommand{\Ll}{\mathcal{L}}
\newcommand{\nN}{\mathcal{N}}

\newtheorem{theorem}{Theorem}[section]
\newtheorem{lemma}[theorem]{Lemma}
\newtheorem{proposition}[theorem]{Proposition}
\newtheorem{definition}[theorem]{Definition}
\newtheorem{example}[theorem]{Example}

\newtheorem{remark}[theorem]{Remark}
\newtheorem{conjecture}[theorem]{Conjecture}

\begin{document}

\title[Weak Zariski decompositions and log terminal models]{Weak Zariski decompositions and log terminal models for generalized polarized pairs}

\begin{abstract}
We show that the existence of a birational weak Zariski decomposition for a pseudo-effective generalized polarized lc pair is equivalent to the existence of a generalized polarized log terminal model.
\end{abstract}

\author{Jingjun Han}
\address{Beijing International Center for Mathematical Research, Peking University,
Beijing 100871, China}
\email{hanjingjun@pku.edu.cn}
\address{Department of Mathematics, Johns Hopkins University, Baltimore, MD 21218, USA}
\email{jhan@math.jhu.edu}

\author{Zhan Li}
\address{Department of Mathematics, Southern University of Science and Technology, 1088 Xueyuan Rd, Shenzhen 518055, China} \email{lizhan@sustc.edu.cn}

\maketitle

\tableofcontents

\section{Introduction}\label{sec: introduction}
Throughout the paper, we work over the complex numbers.

The minimal model conjecture is one of the main problems in birational geometry.	

\begin{conjecture}[Minimal model conjecture]\label{conj:logmmp} For a $\Qq$-factorial dlt pair $(X,B)$, if $K_X+B$ is pseudo-effective, then $(X,B)$ has a log terminal model.
\end{conjecture}

Conjecture \ref{conj:logmmp} is known when $\dim X\le 4$, \cite{Sho09, Birkar10}, or when $(X,B)$ is a $\Qq$-factorial klt pair with a big boundary $B$, \cite{BCHM10}. But when $\dim X\ge 5$, Conjecture \ref{conj:logmmp} is still widely open.

Birkar related the minimal model conjecture to the existence of sections for pseudo-effective adjoint divisors (i.e. the non-vanishing conjecture), \cite{Bir07, Birkar10, Birkar11}. The minimal model conjecture is also related to the existence of Zariski decompositions for adjoint divisors, \cite{Birkarhuweak14}, or even weaker, to the existence of \emph{weak} Zariski decompositions for adjoint divisors, \cite{Birkarweak12}. The later means that $K_X+B\equiv P+N$ where $P$ is a nef divisor and $N$ is an effective divisor. Note that if $K_X+B$ is numerically equivalent to an effective divisor $N$, then it automatically admits a weak Zariski decomposition with $P=0$.

In \cite{Birkarhuweak14}, Birkar and Hu asked that if the existence of weak Zariski decompositions for adjoint pairs implies the minimal model conjecture. This is our starting point of the current paper. However, it seems that the natural setting of Birkar-Hu's question lies in a larger category -- the generalized polarized pairs (g-pairs), and this observation leads to the analogous conjecture on the existence of minimal models for g-pairs (see Conjecture \ref{conj:glogmmp} below). A g-pair $(X, B+M)$ consists of an ``ordinary'' log pair $(X, B)$ plus an auxiliary nef part $M$. A rudimentary model of such pairs had emerged in \cite{Birkarweak12, Birkarhuweak14} but the actual definition only appeared in \cite{BZ16} (although they came from different sources: the former came from the weak Zariski decomposition while the later came from the canonical bundle formula).

A pair (resp. g-pair) is said to be pseudo-effective if $K_X+B$ (resp. $K_X+B+M$) is pseudo-effective. We use ``$/Z$'' to denote a pair relative to the variety $Z$, and the abbreviation ``NQC'' below stands for ``nef $\Qq$-Cartier combinations'' (see Definition \ref{def: NQC}). Also see Definition \ref{def: NQC} for the meaning of a birational NQC weak Zariski decomposition. The NQC pairs are those behave well in the Minimal Model Program (MMP). Hence we state the following conjectures for such pairs.

\begin{conjecture}[Minimal model conjecture for g-pairs]\label{conj:glogmmp}
For a $\Qq$-factorial NQC g-dlt pair $(X/Z,B+M)$, if $K_X+B+M$ is pseudo-effective$/Z$, then $(X/Z,B+M)$ has a g-log terminal model$/Z$.
\end{conjecture}

\begin{conjecture}[Birational weak Zariski decomposition conjecture for g-pairs]\label{conj: NQCbirweakzar} Let $(X/Z,B+M)$ be a $\Qq$-factorial NQC g-dlt pair. If $K_X+B+M$ is pseudo-effective$/Z$, then it admits a birational NQC weak Zariski decomposition$/Z$.
\end{conjecture}

\begin{remark}
By Proposition \ref{prop: g-log minimal model implies NQC zarski decomposition}, Conjecture \ref{conj:glogmmp} implies Conjecture \ref{conj: NQCbirweakzar}.
\end{remark}

Under the NQC assumption, we can answer the aforementioned question of Birkar-Hu in the generalized polarized categories.

\begin{theorem}\label{thm: weak zariski equiv mm}
The birational weak Zariski decomposition conjecture for g-pairs (Conjecture \ref{conj: NQCbirweakzar}) is equivalent to the minimal model conjecture for g-pairs (Conjecture \ref{conj:glogmmp}). 
\end{theorem}

Besides, for Conjecture \ref{conj:glogmmp}, we show that it is a consequence of the termination conjecture for ordinary log pairs. 

\begin{theorem}\label{thm: ter lc implies g-log minimal model}
Assume the termination of flips for $\Qq$-factorial dlt pairs$/Z$, then any pseudo-effective $\Qq$-factorial NQC g-dlt pair $(X/Z,B+M)$ has a g-log terminal model. 
\end{theorem}

As for the termination of g-MMP, we show the following result.

\begin{theorem}\label{thm: weak zariski scaling ample}
Assume the birational weak Zariski decomposition conjecture (Conjecture \ref{conj: NQCbirweakzar}). Let $(X/Z,B+M)$ be a pseudo-effective $\Qq$-factorial NQC g-dlt pair. Then any sequence of g-MMP on $(K_X+B+M)$ with scaling of an ample divisor$/Z$ terminates.
\end{theorem}

We briefly describe the idea of the proof of Theorem \ref{thm: weak zariski equiv mm}. The nontrivial part is to show that Conjecture \ref{conj: NQCbirweakzar} implies Conjecture \ref{conj:glogmmp}. Suppose that $K_X+B+M \equiv P+N$ is a weak Zariski decomposition with $P$ nef and $N \geq 0$. As it was pointed out in \cite[\S 6]{Birkarhuweak14}, one only needs to consider the case $\Supp N \subseteq \Supp\lf B\rf$. In this case, we run a special g-MMP on $(K_X+B+M)$ with scaling of a divisor (not necessarily effective) such that the g-MMP is $N$-negative and in each step, $N_i+\nu_i P_i$ is nef, where
\[
\nu_i \coloneqq \inf\{t \geq 1 \mid N_i+tP_i\text{~is nef}\}
\] is the nef threshold. Notice that once $\nu_i=1$, then $K_{X_i}+B_i+M_i \equiv P_i+N_i$ is nef, and we are done. Thus the idea is to decrease $\nu_i$ until it reaches $1$. In fact, as long as $\nu_i>1$, we can get a smaller $\nu_{i+1}$. After this, we use the special termination and the induction on dimensions to conclude that the sequence $\{v_i\}$ cannot be infinite. The special termination still holds for this setting (see Theorem \ref{prop: special termination 2}) by an appropriate adaptation of the argument in \cite{Birkar12}. Finally, by putting these together, we can prove the result.

\begin{remark}
By different approaches, Hacon and Moraga independently obtain some of the above results in certain general forms, \cite{HM18weak}. They use ideas from \cite{Bir07}, while our proof is based on ideas in \cite{Birkar11,Birkar12,Birkarhuweak14}. Furthermore, our results are about boundaries with real coefficients, while results in \cite{HM18weak} require the boundaries with rational coefficients.\end{remark}

The paper is organized as following. In Section \ref{sec: preliminaries}, we collect the relevant definitions. In Section \ref{sec: LMMP}, we first elaborate on the MMP for g-pairs (g-MMP) developed in \cite{BZ16}, and then prove some standard results in this setting. We also introduce the g-MMP with scaling of an NQC divisor. In Section \ref{sec: special termination}, we establish the special termination result for g-MMP with scaling. The proofs of the theorems are given in Section \ref{sec: proof}.

\medskip

\noindent\textbf{Acknowledgements}.
We would like to thank Caucher Birkar, the paper is deeply influenced by his ideas. We thank Chen Jiang for showing us a simple proof of the length of extremal rays for the g-pairs. We also thank the participants of the Birational Geometry Seminar at BICMR/Peking University for their interests in the work, especially Yifei Chen and Chuyu Zhou. J. H. thanks Caucher Birkar for the invitation to the University of Cambridge where the paper was written, the thanks also go to his advisors Gang Tian and Chenyang Xu for constant support and encouragement. Z. L. thanks Keiji Oguiso for the invitation to the University of Tokyo where some key ideas are conceived. This work is partially supported by NSFC Grant No.11601015.

\section{Preliminaries}\label{sec: preliminaries}
\subsection{Generalized polarized pairs}

\begin{definition}[Generalized polarized pair]\label{def: g-pair}
	A \emph{generalized polarized pair (g-pair)} over $Z$ consists of a normal variety $X$ equipped with projective morphisms 
	\[
	\tilde X \xrightarrow{f}X \to Z,
	\] where $f$ is birational and $\tilde X$ is normal, an $\Rr$-boundary $B \geq 0$ on $X$, and an $\Rr$-Cartier divisor $\tilde M$ on $\tilde X$ which is nef$/Z$ such that $K_{X}+B+M$ is $\Rr$-Cartier, where $M\coloneqq f_*\tilde M$. We say that $B$ is the boundary part and $M$ is the nef part. 
\end{definition}
	
For our convenience, when the base $Z$, the boundary part and the nef part are clear from the context, we will just say that $(X,B+M)$ is a g-pair. Notice that, in contrast to \cite{BZ16}, we denote the generalized polarized pair by $(X, B+M)$ instead of $(X', B'+M')$.

Let $g:X'\to \tilde{X}$ be a birational morphism, 
such that $X'\to X$ is a log resolution of $(X,B)$. Let
$$K_{X'}+B'+M'=g^{*}(K_X+B+M),$$
where $M'=g^{*}\tilde{M}$. 
We say that $(X',B'+M')\to X$ is a \emph{log resolution} of $(X,B+M)$. By replacing $(\tilde{X},\tilde{B}+\tilde{M})$ with $(X', B'+M')$, we may assume that $\tilde X\to X$ is a log resolution of $(X,B+M)$. In the same fashion, $\tilde X$ can be chosen as a sufficiently high model of $X$. In particular, if there exists a variety $Y$ birational to $X$, we can always assume that there exists a morphism from $\tilde X$ to $Y$ which commutes with $X \dasharrow Y$.

Many definitions/notions for ordinary log pairs have counterparts for generalized polarized pairs. For convenience, we use the prefix ``g-'' to denote the corresponding notions. For example, one can define the \emph{generalized log discrepancy} (g-log discrepancy) of a prime divisor $E$ over $X$: let $\tilde X$ be a high enough model which contains $E$, and let
\[
K_{\tilde X}+\tilde B+\tilde M=f^*(K_{X}+B+M).
\] Then the g-log discrepancy of $E$ is defined as (see \cite{BZ16} Definition 4.1)
\[
a(E, X, B+M)=1-{\rm mult}_E\tilde B.
\] 
A g-lc place is a divisor $E$ on a birational model of $X$, such that $a(E; X, B+M)=0$. A g-lc center is the image of a g-lc place, and the g-lc locus is the union of all the g-lc centers.

\medskip

We say that $(X,B+M)$ is generalized lc (g-lc) (resp. generalized klt (g-klt)) if the g-log discrepancy of any prime divisor is $\geq 0$ (resp. $>0$). 

Moreover, as $\tilde M$ is a nef divisor, if $M$ is $\Rr$-Cartier, by the negativity lemma (see \cite[Lemma 3.39]{KM98}), $f^* M=\tilde M+E$ with $E \geq 0$ an exceptional divisor. In particular, this implies that if $K_{X}+B$ is $\Rr$-Cartier, then the log discrepancy of a divisor $E$ with respect to $(X, B)$ is greater than or equal to the g-log discrepancy of $E$ with respect to $(X, B+M)$.

The definition of generalized dlt (g-dlt) is subtle. 
\begin{definition}[G-dlt]\label{def:g-dlt}
Let $(X,B+M)$ be a g-pair. We say that $(X,B+M)$ is \emph{g-dlt} if it is g-lc and there is a closed subset $V\subset X$ ($V$ can be equal to $X$) such that
\begin{enumerate}
\item $X\backslash V$ is smooth and $B|_{X\backslash V}$ is a simple normal crossing divisor,
\item if $a(E, X, B+M)=0$, then the center of $E$ satisfies $\Center_X(E) \not\subset V$ and $\Center_X(E)\backslash V$ is a lc center of $(X\backslash V, B\backslash V)$.
\end{enumerate}
\end{definition}
\begin{remark}\label{rmk: klt}
 If $(X,B+M)$ is a $\Qq$-factorial dlt pair, then $X$ is klt.
\end{remark}

Our definition of g-dlt is slightly different from the definition in \cite[page 13]{Bir16a}. We will show that our definition of g-dlt is preserved under adjunctions and running MMPs. 
\begin{remark}\label{rmk: dlt}
	Another possible definition of g-dlt is as follows: 
	
	A g-pair $(X,B+M)$ is g-dlt if there exists a log resolution $\pi: X'\to X$ of $(X,B+M)$, such that $a(E,X,B+M)>0$ for every exceptional divisor $E\subset Y$.
	
	For the ordinary log pairs (i.e. $\tilde M=0$), the above two definitions are equivalent (see \cite[Theorem 2.44]{KM98}). However, for g-pairs, it is not known whether the two definitions are the same or not.
\end{remark}

The adjunction formula for g-lc pairs is given in \cite[Definition 4.7]{BZ16}. 

\begin{definition}[Adjunction formula for g-pairs]\label{def: g-adjunction}
	Let $(X/Z,B+M)$ be a g-dlt pair with data $\tilde X \xrightarrow{f} X \to Z$ and $\tilde M$. Let $S$ be a component of $\lfloor B \rfloor$ and $\tilde S$ its birational transform on $\tilde X$. We may assume that $f$ is a log resolution of $(X,B+M)$. Write 
	\[
	K_{\tilde X} +\tilde B+\tilde M=f^*(K_{X} +B +M),
	\] then
	\[
	K_{\tilde S} +B_{\tilde S} +M_{\tilde S} \coloneqq (K_{\tilde X} +\tilde B+\tilde M)|_{\tilde S}
	\]
	where $B_{\tilde S} = (\tilde B-\tilde S)|_{\tilde S}$ and $M_{\tilde S} =\tilde M|_{\tilde S}$. Let $g$ be the induced morphism $\tilde S\to S$. Set $B_{S} = g_*B_{\tilde S}$ and $M_{S} =g_*M_{\tilde S}$. Then we get the equality
	\[
	K_{S}+B_{S}+M_{S} = (K_{X}+B+M)|_{S},
	\] which is referred as \emph{generalized adjunction formula}.
\end{definition}
	
	Suppose that $\tilde M=\sum \mu_i \tilde M_i$, where $\tilde M_i$ is a nef$/Z$ Cartier divisor for each $i$, and $B=\sum b_j B_j$ the prime decomposition of the $\Rr$-divisor $B$. Let $\bm{b}=\{b_j\}$, $\bm{\mu}=\{\mu_i\}$ be the coefficient sets. For a set of real numbers $\Gamma$, set 
\begin{equation}\label{eq: S}
\mathbb{S}(\Gamma) \coloneqq \{1-\frac{1}{m}+\sum_{j} \frac{r_j\gamma_j}{m}\le 1 \mid m\in\mathbb{Z}_{>0},r_j\in\Zz_{\ge0}, \gamma_j \in \Gamma\}\cup\{1\}.
\end{equation}
Then the coefficients of $B_{S}$ belong to the set $\mathbb{S}(\bm{b},\bm{\mu}) \coloneqq \mathbb S(\bm{b} \cup \bm{\mu})$ (see \cite[Proposition 4.9]{BZ16}).

\begin{lemma}\label{lem:adjgdlt}
	Let $(X/Z,B+M)$ be a g-dlt pair with data $\tilde X \xrightarrow{f} X \to Z$ and $\tilde M$. Let $S$ be a component of $\lfloor B \rfloor$, and $B_S,M_S$ be the divisors in the g-adjunction formula (see Definition \ref{def: g-adjunction}
). Then $(S,B_S+M_S)$ is still g-dlt. 
\end{lemma}

\begin{proof}
	We use the notation in Definition \ref{def: g-adjunction}. Let $V$ be the closed subset $V\subset X$ in Definition \ref{def:g-dlt}, and $V_S=V\cap S$. It is clear that $S\backslash V_S$ is smooth and $B|_{S\backslash V_S}$ is a simple normal crossing divisor. 
	
	If $a(E,S,B_S+M_S)=0$, then $\Center_{\tilde S}(E)$ is a stratum of $(\tilde{S},\tilde{B}_{\tilde S})$, and thus a stratum of $({\tilde X}, \tilde B)$. Let $E'$ be a g-lc place of $(X,B+M)$, such that $\Center_{\tilde S}(E)=\Center_{\tilde X}(E')$. Since $(X,B+M)$ is g-dlt, $\Center_{X}(E') \not\subset V$ and $\Center_{X}(E')\backslash V$ is a lc center of $(X\backslash V,B\backslash V)$. Thus, $\Center_{S}(E)\backslash V_S$ is a lc center of $(S\backslash V_S,B_S|_{S\backslash V_S})$.
\end{proof}

\begin{remark}
	In general, $K_S+B_S$ may not be $\Rr$-Cartier and thus $(S,B_S)$ may not be dlt. In particular, $(S, B_S+M_S)$ may not be g-dlt in the sense of \cite{Bir16a}. This is the main reason that we do not use the definition of g-dlt as \cite{Bir16a}.
\end{remark}

The following proposition is similar to \cite[Proposition 3.9.2]{Fujino07}.

\begin{proposition}\label{prop: intersection on g-lc centers}
	Let $(X/Z,B+M)$ be a g-dlt pair with data $\tilde X \xrightarrow{f} X \to Z$ and $\tilde M$. Suppose that $\tilde M=\sum \mu_i \tilde M_i$, where $\tilde M_i$ is a nef$/Z$ Cartier divisor for every $i$, and let $B=\sum b_j B_j$ be the prime decomposition of an $\Rr$-divisor $B$. Let $V$ be a g-lc center of $(X,B+M)$. Then there exists a g-dlt pair $(V, B_V+M_V)$ such that 
	\[
	(K_X+B+M)|_{V}=K_V+B_{V}+M_{V},
	\]
	where $M_{V}$ is the push forward of $\tilde{M}|_{\tilde{V}}$ on $V$, and $\tilde{V}$ is the birational transform of $V$ on $\tilde{X}$. Moreover, the coefficients of $B_{V}$ belong to the set $\mathbb{S}(\bm{b},\bm{\mu})$.
\end{proposition}
\begin{proof}
	Let $k$ be the codimension of $V$. By definition of g-dlt, $V$ is an irreducible component of $V_{1}\cap V_{2}\cap \ldots\cap V_{k}$ for some $V_{i}\subseteq  \lfloor B\rfloor$. Under the notation of \eqref{eq: S}, a straightforward computation shows that $\mathbb{S}(\bm{b},\bm{\mu})=\mathbb{S}(\mathbb{S}(\bm{b},\bm{\mu}))$ (for example, see \cite[Proposition 3.4.1]{HMX14}). Then the claim follows from applying Lemma \ref{lem:adjgdlt} $k$ times. 
\end{proof}

\subsection{G-log minimal models and g-log terminal model}\label{sec: MM and Mori fibre space}

The notions of log minimal/terminal models still make sense in the generalized polarized setting. Following Shokurov \cite{Sho92}, certain extractions of g-lc places are allowed for g-log minimal models. First, if $f: X \dashrightarrow Y$ is a birational map, and $B$ is an effective divisor on $X$, we define
\begin{equation}\label{eq: B_Y}
B_Y \coloneqq f_*B +E,
\end{equation} where $f_*B$ is the birational transform of $B$ on $Y$, and $E$ is the sum of reduced exceptional divisors of $f^{-1}$. 

\begin{definition}[G-log minimal model \& g-log terminal model]\label{def: g-log minimal model and g-log terminal model}
	Let $(X/Z, B+M)$ be a g-pair with data $\tilde X \to X$ and nef part $M$. Then a g-pair $(Y/Z, B_Y+M_Y)$ is called a \emph{g-log minimal model} of $(X/Z, B+M)$, if
	\begin{enumerate}
		\item there is a birational map $X \dashrightarrow Y$,
		\item $B_Y$ is the same as \eqref{eq: B_Y}, and $M_Y$ is the pushforward of $\tilde M$ (we can always assume that there exists a morphism $\tilde X \to Y$),
		\item $K_Y+B_Y+M_Y$ is nef, 
		\item $(Y/Z, B_Y+M_Y)$ is $\Qq$-factorial g-dlt with data $\tilde X \to Y$ and nef part $M_Y$,
		\item $a(D, X, B+M)<a(D,Y, B_Y+M_Y)$ for any divisor $D$ on $X$ which is exceptional over $Y$.
	\end{enumerate}
	Furthermore, if $X \dasharrow Y$ is a birational contraction (i.e. there is no divisor on $Y$ which is exceptional over $X$), we say that $(Y/Z, B_Y+M_Y)$ is a \emph{g-log terminal model} of $(X/Z, B+M)$.
\end{definition}

\begin{remark}\label{rmk: extract lc places}
Just as the case for log pairs, a g-log minimal model can only extracts g-lc places. That is, a divisor $E\subset Y$ is exceptional over $X$ satisfies $a(E, X, B+M)=0$. 
\end{remark}

\subsection{Weak Zariski Decompositions}\label{sec: Zariski decomposition}
On a normal variety $X$ over $Z$ (we write this by $X/Z$), an $\Rr$-Cartier divisor $D$ is said to admit a \emph{weak Zariski decomposition} if 
\[D \equiv N+P/Z,\] with $N \geq 0$ and $P$ a nef$/Z$ $\Rr$-Cartier divisor (see \cite[Definition 1.3]{Birkarweak12}). Unlike Zariski decompositions, weak Zariski decompositions may not be unique.

\begin{definition}\label{def: bir zariski decomposition}
	Let $X/Z$ be a variety and $D$ be a Cartier divisor. We say that $D$ admits a \emph{birational Zariski decomposition} if there exists a birational morphism $f: Y \to X$ from a normal variety $Y$, such that $f^*D$ admits a weak Zariski decomposition. 
\end{definition}

Notice that birational weak Zariski decomposition is called weak Zariski decomposition in \cite[Definition 1.3]{Birkarweak12}. For a lc pair $(X, \Delta)$, the non-vanishing conjecture asserts that $K_X+\Delta \sim_\Rr N$ for some effective divisor $N$. This implies that  $K_X+\Delta$ admits a weak Zariski decomposition by taking $P=0$.

For (weak) Zariski decompositions, the most important case is when $D$ equals to the adjoint divisor $K_X+B$ (or $K_X+B+M$). In the sequel, when saying ``existence a (weak) Zariski decomposition'' without referring to a divisor, it should be understood that the decomposition is for the adjoint divisor.

\begin{remark}
	Zariski proved that on a smooth projective surface, an effective divisor $D$ can be decomposed as a sum of a nef divisor and an effective divisor with some extra properties, \cite{Zariski62}, which is known as the Zariski decomposition. There are various generalizations to higher dimensions. See \cite{Nakayamazariski} for more details. For an arbitrary divisor, it may not always admit a weak Zariski decomposition, \cite{John14}. But for the adjoint divisor $K_X+B$, the existence of a weak Zariski decomposition is a consequence of the existence of a minimal model, which is highly plausible. 
\end{remark}

\subsection{Nef $\Qq$-Cartier combinations (NQC)}\label{sec: NQC}

We need a technical assumption to guarantee that certain g-MMP on g-pairs behave as the ordinary log pairs (see Section \ref{subsection: length of extremal rays and its applications}). Here the abbreviation ``NQC'' stands for ``nef $\Qq$-Cartier combinations''. 

\begin{definition}\label{def: NQC} 
	We have following definitions on decompositions of nef$/Z$ $\Rr$-Cartier divisors in various settings. 
	\begin{enumerate}
		\item We say that an $\Rr$-Cartier divisor $M$ is \emph{NQC} over $Z$, if
		\[
		M\equiv\sum_i r_i M_i/Z,
		\] where $r_i \in \Rr_{>0}$ and $M_i$ are $\Qq$-Cartier nef$/Z$ divisors. 
		\item A g-pair $(X/Z, B+M)$ with data $\tilde X \xrightarrow{f}X \to Z$ and $\tilde M$ is said to be an \emph{NQC g-pair}, if $\tilde M$ is NQC. 
		\item We define \emph{NQC g-lc, NQC g-klt}, etc. if an NQC g-pair is g-lc, g-klt, etc.
		\item We say that a g-pair $(X/Z, B+M)$ admits a \emph{birational NQC weak Zariski decomposition$/Z$}, if there exists a birational morphism $g: Y \to X/Z$ such that $g^*(K_X+B+M)\equiv P+N/Z$, where $N \geq 0$ and $P$ is NQC over $Z$.
	\end{enumerate}
\end{definition} 

We will avoid repeating ``over $Z$'' if the base $Z$ is clear in the context. By definition, the NQC property is preserved under g-MMPs and generalized adjunctions. 

The NQC assumption excludes the pathological phenomenon incurred  by the nef part. In \cite{BZ16}, Birkar proved ACC for g-lc thresholds and Global ACC for NQC pairs (though the name was not given there). In the current paper, we need the g-lc pairs to be NQC in order to run some kind of special g-MMPs.

In Proposition \ref{prop: g-log minimal model implies NQC zarski decomposition}, we show that the existence of a g-log minimal model for an NQC g-lc pair implies the existence of a birational NQC weak Zariski decomposition for this g-pair.

\section{MMP for generalized polarized pairs}\label{sec: LMMP}

\subsection{MMP for generalized polarized pairs}\label{subsection: MMP for g-pairs}
For g-lc pairs, the cone theorem, the existence of flips and the termination of flips are still expected to hold true. 
\begin{conjecture}[Cone theorem for g-lc pairs]
	Let $(X,B+M)$ be a g-lc pair. We have the following claims.
	\begin{enumerate}
		\item There are countably many curves $C_j\subset X$ such that $0<-(K_X+B+M)\cdot C_j\le 2\dim X$, and
		$$\overline{\rm NE}(X)=\overline{\rm NE}(X)_{(K_X+B+M)\ge0}+\sum\Rr_{\ge0}[C_j].$$
		\item Let $F\subset \overline{\rm NE}$ be a $(K_X+B+M)$-negative extremal face. There there a unique morphism ${\rm cont}_{F}:X\to Y$ to a projective variety such that  $({\rm cont}_{F})_{*}\Oo_X=\Oo_Y$, and an irreducible curve $C\subset X$ is mapped to a point by ${\rm cont}_{F}$ if and only if $[C]\in F$. 
		\item Let $F$ and ${\rm cont}_{F}:X\to Y$ be as in $(2)$. Let $L$ be a line bundle on $X$ such that $L\cdot C=0$ for every curve $C$ with $[C]\in F$. Then there is a line bundle $L_Y$ on $Y$ such that $L\simeq {\rm cont}_{F}^{*}L_Y$. 
	\end{enumerate}
\end{conjecture}

\begin{definition}[Flips for g-pairs]
Let $(X,B+M)$ be a g-lc pair. A \emph{$(K_X+B+M)$-flipping contraction} is a proper birational morphism $f:X\to Y$ such that $\Exc(f)$ has codimension at least two in $X$, $-(K_X+B+M)$ is $f$-ample and the relative Picard group has rank $\rho(X/Y)=1$.
	
	A g-lc pair $(X^{+},B^{+}+M^{+})$ together with a proper birational morphism $f^{+}:X^{+}\to Y$ is called a \emph{$(K_X+B+M)$-flip} of $f$ if 
	\begin{enumerate}
		\item $B^{+}, M^{+}$ are the birational transforms of $B, M$ on $X^{+}$, respectively,
		\item $K_{X^{+}}+B^{+}+M^{+}$ is $f^{+}$-ample, and
		\item $\Exc(f^{+})$ has codimension at least two in $X^{+}$.
	\end{enumerate}
	For convenience, We call the induced birational map, $X\dashrightarrow X^{+}$, a $(K_X+B+M)$-flip.
\end{definition}
 
\begin{conjecture}[Existence of flips for g-lc pairs] For a g-lc pair, the flip of a flipping contraction always exists.
\end{conjecture}
\begin{conjecture}[Termination of flips for g-lc pairs] There is no infinite sequence of flips for g-lc pairs.
\end{conjecture}

Although the MMP for g-pairs is not established in full generality, some important cases could be derived from the standard MMP. We elaborate these results which are developed in \cite[\S 4]{BZ16}.

Let $(X/Z,B+M)$ be a $\Qq$-factorial g-lc pair, and $A$ be a general ample$/Z$ divisor. 

\begin{quote}
	($\star$) Suppose that for any $0<\epsilon\ll1$, there exists a boundary $\Delta_{\epsilon} \sim_\Rr B+M+\epsilon A/Z$, such that $(X,\Delta_{\epsilon})$ is klt. 
\end{quote}

Under the assumption ($\star$), we can run a g-MMP$/Z$ on $(K_X+B+M)$, although the termination is not known. In fact, let $R$ be an extremal ray$/Z$, such that $(K_X+B+M)\cdot R<0$. For $0<\epsilon\ll 1$, we have $(K_X+B+M+\epsilon A)\cdot R<0$. By assumption, there exists $\Delta_{\epsilon} \sim_\Rr B+M+\epsilon A/Z$, such that $(X,\Delta_{\epsilon})$ is klt, and $(K_X+\Delta_{\epsilon})\cdot R<0$. Now, $R$ can be contracted and its flip exists if the contraction is a flipping contraction. If we obtain a g-log minimal model or a g-Mori fiber space, we stop the process. Otherwise, let $X\dashrightarrow Y$ be the divisorial contraction or the flip, then $(Y,B_{Y}+M_{Y})$ is naturally a g-lc pair. Moreover, for any $0<\epsilon\ll 1$, $(Y,\Delta_{\epsilon,Y})$ is klt. Repeating this process gives the g-MMP$/Z$.

The usual notion of g-MMP with scaling of the general divisor $A$ also makes sense under assumption ($\star$) (see \cite{BZ16}).

The following lemma shows that assumption ($\star$) is satisfied in two cases. As a result, we may run a g-MMP for such g-pairs.

\begin{lemma}\label{lem:glcklt}
	Let $(X/Z,B+M)$ be a g-lc pair, and $A$ be an ample$/Z$ divisor. Suppose that either 
	
	(i) $(X, B+M)$ is g-klt, or
	
	(ii) there exists a boundary $C$, such that $(X,C)$ is klt.
	
	Then, there exists a boundary $\Delta \sim_\Rr B+M+A/Z$, such that $(X,\Delta)$ is klt. 
	
	Moreover, if $X$ is $\Qq$-factorial, we may run a g-MMP on $K_X+B+M$.
\end{lemma}

\begin{proof}
	Suppose that $(X, B+M)$ is g-klt. We have
	\[
	K_{\tilde X}+\tilde B+\tilde M+f^*(A) =f^*(K_{X}+B+M+A),
	\] 
	where $\tilde M+f^*(A)$ is big and nef. Hence for $k \in \Nn$, there exists an ample divisor$/Z$ $H_k\ge0$, and an effective divisor $E$, such that $\tilde{M}+f^*(A)\sim_\Rr H_k+\frac {1}{k}E$. For general $H_k$ and $k \gg 1$, $K_{\tilde X}+\tilde B+H_k+\frac{1}{k}E$ is sub-klt.
	 Let 
	\[
	\Delta\coloneqq f_*(\tilde B+H_k+\frac 1 k E)\sim_\Rr B+M+A.
	\] 
	Since 
	\[
	K_{\tilde X}+\tilde B+\tilde M+f^*(A) \sim_\Rr K_{\tilde X}+\tilde B+H_k+\frac{1}{k}E,
	\] 
	we have 
	\[
	K_{\tilde X}+\tilde B+H_k+\frac{1}{k}E = f^*(K_{X}+\Delta),
	\] and $(X,\Delta)$ is klt. 
	
	Suppose that (ii) holds. By assumption, $B+M-C$ is $\Rr$-Cartier, and there exists $0<\epsilon\ll1$, such that $\epsilon(B-C+M)+A/2$ is ample. Let $H$ be a general ample divisor, such that $\epsilon H\sim_{\Rr}\epsilon(B-C+M)+A/2$, and $(X,C+H)$ is klt. Thus $(X, (\epsilon C+(1-\epsilon) B)+(\epsilon H+(1-\epsilon)M))$ is g-klt with boundary part $\epsilon C+(1-\epsilon) B$ and nef part $\epsilon H+(1-\epsilon)M$. Besides,
	\begin{align*}
	&K_X+(\epsilon C+(1-\epsilon) B)+(\epsilon H+(1-\epsilon)M)\\
	\sim_{\Rr}&\epsilon(K_X+C+H)+(1-\epsilon)(K_X+B+M)\\
	\sim_{\Rr}&K_X+B+M+A/2.	
	\end{align*}
	Apply (i) to $(X, (\epsilon C+(1-\epsilon) B)+(\epsilon H+(1-\epsilon)M))$ with $A/2$ we get (ii).
\end{proof}
\begin{remark}
	As a simple corollary, suppose that $X$ is $\Qq$-factorial klt and $(X/Z,B+M)$ is g-lc, then there are countably many extremal rays $R/Z$, such that $(K_X+B+M)\cdot R<0$. 
\end{remark}

A g-MMP for a g-dlt pair preserves g-dltness. 
\begin{lemma}\label{lem:dltpreserved}
	Let $(X/Z,B+M)$ be a $\Qq$-factorial g-dlt pair. Let $g: X\dashrightarrow Y/Z$ be either a divisorial contraction of a $(K_X+B+M)$-negative extremal ray or a $(K_X+B+M)$-flip. Let $B_Y=g_{*}(B), M_Y=g_{*}(M)$, then $(Y,B_Y+M_Y)$ is also $\Qq$-factorial g-dlt.
\end{lemma}
\begin{proof}
	Fix a general ample$/Z$ divisor $H$. As $X$ is klt (see Remark \ref{rmk: klt}), by Lemma \ref{lem:glcklt}, there exist $0<\epsilon\ll 1$ and $\Delta_{\epsilon}$, such that $\Delta_{\epsilon}\sim_{\Rr} B+M+\epsilon H$, and $(X,\Delta_{\epsilon})$ is $\Qq$-factorial klt. Moreover, $g$ is also either a divisorial contraction of a $(K_X+\Delta_{\epsilon})$-negative extremal ray or a $(K_X+\Delta_{\epsilon})$-flip. By \cite[Proposition 3.36, 3.37]{KM98}, $Y$ is $\Qq$-factorial. 
	
To show that the g-dltness is preserved, we use a similar argument as \cite[Lemma 3.44]{KM98}. Let $V$ be the closed subset $V\subset X$ in Definition \ref{def:g-dlt}. When $g$ is a divisorial contraction, set $V_Y= g(V \cup \Exc(g))$. Then $X \backslash (V\cup \Exc(g)) \simeq Y \backslash V_Y$, and thus $Y\backslash V_Y$ is smooth and $B|_{Y\backslash V_Y}$ is a simple normal crossing divisor. By the negativity lemma, for any divisor $E$, we have $a(E, Y, B_Y+M_Y) \geq a(E, X, B+M) \geq 0$, and $a(E, Y, B_Y+M_Y)> a(E, X, B+M)$ when $\Center_X(E)\subset \Exc(g)$ (see \cite[Lemma 3.38]{KM98}). When $a(E, Y, B_Y+M_Y)=0$, we have $a(E, X, B+M)=0$ and $\Center_X(E)\not\subset \Exc(g)$. Thus $\Center_X(E) \not\subset V \cup \Exc(g)$ by the definition of g-dlt. Hence $\Center_Y(E) \not\subset V_Y$. This shows that $(Y, B_Y+M_Y)$ is g-dlt.
	
	When $g$ is a flip, let $c_1: X \to W/Z, c_2: Y \to W/Z$ be the corresponding contractions. Let $L_1, L_2$ be the contraction loci of $c_1, c_2$ respectively. Set $V_X = V \cup L_1 \cup c_1^{-1}(c_2(L_2))$ and $V_Y = c_2^{-1}(c_1(V)) \cup L_2 \cup c_2^{-1}(c_1(L_1))$. Then $X \backslash V_X \simeq Y \backslash V_Y$, and thus $Y\backslash V_Y$ is smooth and $B|_{Y\backslash V_Y}$ is a simple normal crossing divisor. By the negativity lemma, for any divisor $E$, we have $a(E, Y, B_Y+M_Y) \geq a(E, X, B+M) \geq 0$, and $a(E, Y, B_Y+M_Y)> a(E, X, B+M)$ when $\Center_X(E)\subset L_1$ or $\Center_Y(E)\subset L_2$ (see \cite[Lemma 3.38, 3.44]{KM98}). When $a(E, Y, B_Y+M_Y)=0$, we have $a(E, X, B+M)=0$ and $\Center_Y(E)\not\subset L_2 \cup c_2^{-1}(c_1(L_1))$. Besides, by the definition of g-dlt, $\Center_X(E) \not\subset V$. Thus $\Center_Y(E) \not\subset V_Y$. This shows that $(Y, B_Y+M_Y)$ is g-dlt.
\end{proof}

Let $f: X \to Y$ be a projective morphism of varieties and $D$ be an $\Rr$-Cartier divisor on $X$, then $D$ is called a \emph{very exceptional} divisor of $f$ (\cite[Definition 3.2]{Sho03}) if (1) $f(D) \subsetneq Y$, and (2) for any prime divisor $P$ on $Y$ there is a prime divisor $Q$ on $X$ which is not a component of $D$ but $f(Q)=P$. Notice that if $f$ is birational, then any exceptional divisor is very exceptional. The point is that the negativity lemma also holds for very exceptional divisors (\cite[Lemma 3.3]{Birkar12}). The following Proposition is an easy consequence of the negativity lemma, which is a generalization of \cite[Theorem 3.5]{Birkar12} in the setting of g-lc pairs.

\begin{proposition}\label{prop: termination for very exceptional divisor}
	Let $(X/Z, B+M)$ be a g-lc pair with $X$ klt such that $K_X+B+M\equiv D_1-D_2$, where $D_1\geq 0,D_2\ge0$ have no common components. Suppose that $D_1$ is a very exceptional divisor over $Z$. Then any g-MMP$/Z$ on $K_X+B+M$ with scaling of an ample divisor$/Z$ either terminates to a Mori fiber space or contracts $D_1$ after finite steps. Moreover, if $D_2=0$, then the g-MMP terminates to a model $Y$ such that $K_Y+B_Y+M_Y \equiv 0/Z$.
\end{proposition}
\begin{proof}
	Let $A$ be an ample$/Z$ divisor, by Lemma \ref{lem:glcklt}, we can run a g-MMP$/Z$ with scaling of $A$. Let $\nu_i=\inf\{t \in\Rr_{\geq0} \mid K_{X_i}+B_i+M_i+tA_i \text{~is~nef~over~}Z\}$ be the nef threshold in each step of the scaling. Set $\mu = \lim \nu_i$. 
	
	If $\mu>0$, then the g-MMP is also a g-MMP on $K_X+B+M+\mu A$. By Lemma \ref{lem:glcklt}, there exists a boundary $\Delta$, such that $K_X+B+M+\mu A \sim_\Rr K_X+\Delta$, and $(X, \Delta)$ is klt with $\Delta$ big. Then the $(K_X+\Delta)$-MMP with scaling terminates by \cite[Corollary 1.4.2]{BCHM10}. In this case, without loss of generality, we can assume that $K_X+B+M$ is nef$/Z$.
	
	 If $\mu=0$, after finite steps, we can assume that the g-MMP only consists of flips. Thus $K_X+B+M$ is a limit of movable$/Z$ $\Rr$-Cartier divisors. 
	 
	 In either case, for any prime divisor $S$ on $X$, and a very general curve $C\subset S$, we have $(K_X+B+M)\cdot C =(D_1-D_2)\cdot C\geq 0$. Since $D_1$ is a very exceptional divisor over $Z$, by the negativity lemma (\cite[Lemma 3.3]{Birkar12}), $D_2-D_1 \geq 0$, which implies that $D_1=0$. Hence the g-MMP contracts $D_1$ after finite steps. When $D_2=0$, on a model $Y$ such that $D_1$ is contracted, we have $K_Y+B_Y+M_Y \equiv 0/Z$.
\end{proof}
The proof of \cite[Lemma 4.5]{BZ16} gives the existence of g-dlt modifications. 
\begin{proposition}[G-dlt modification {\cite[Lemma 4.5]{BZ16}}]\label{prop: dlt}
	Let $(X, B + M)$ be a g-lc pair with data $\tilde X \xrightarrow{f} X \to Z$ and $\tilde M$. Then perhaps after replacing $f$ with a higher resolution, there exist a $\Qq$-factorial g-dlt pair $(X', B' + M')$ with data $\tilde X \xrightarrow{g} X' \to Z$ and $\tilde M$, and a projective birational morphism $\phi: X' \to X$ such that $K_{X'}+B'+M' = \phi^*(K_{X}+B+M)$. Moreover, each exceptional divisor of $\phi$ is a component of $\lf B'\rfloor$. We call $(X', B' + M')$ a g-dlt modification of $(X,B+M)$.
\end{proposition}
\begin{proof}
	We may assume that $f$ is a log resolution of $(X,B+M)$. Let $E_i$ be an irreducible exceptional divisor of $f$. We have
\[
K_{\tilde{X}}+\tilde{B}+E+\tilde{M}=f^*(K_{X}+B+M)+E\equiv E/X,
\] where $E=\sum a_iE_i\ge0 $, and $a_i=a(E_i, X, B+M)$ is the g-log discrepancy. We can run a g-MMP$/X$ on $(K_{\tilde{X}}+\tilde{B}+E+\tilde{M})$ with scaling of an ample divisor. By Proposition \ref{prop: termination for very exceptional divisor}, the g-MMP terminates to $X'$, and $E$ is contracted. Thus $K_{X'}+B'+M'=\phi^*(K_{X}+B+M)$. As $(\tilde{X}, \tilde{B}+E+\tilde{M})$ is g-dlt, by Lemma \ref{lem:dltpreserved}, $(X', B' + M')$ is also $\Qq$-factorial g-dlt. By the construction of $E$, we see that each exceptional divisor of $\phi$ is a component of $\lf B'\rfloor$
\end{proof}

Although the MMP is expected to hold for g-pairs, abundance conjecture, finite generations of canonical rings and non-vanishing conjecture all fail for g-pairs (see \cite[\S 3]{Birkarhuweak14} for discussions). However, as for non-vanishing conjecture, one can still ask under the numerical sense. In general, abundance conjecture does not hold under the numerical sense.
\begin{example}
	Let $X$ be $\Pp^2$ blown up 9 points in sufficiently general positions, $M=-2K_X$. Then $K_X+M=-K_X$ is nef, but there is no semiample divisor $N$ such that $K_X+M\equiv N$,  \cite{bauer2004}. 
\end{example}
We ask the following question.
\begin{conjecture}[Weak non-vanishing \& weak abundance for g-pairs]
	Let $(X/Z,B+M)$ be a $\Qq$-factorial NQC g-dlt pair. Suppose that $K_X+B+M$ is nef, then 
	
	(1) there exists an effective $\Rr$-divisor $N$, such that $K_X+B+M\equiv N$,
	
	(2) there exists $0\le t\le 1$, and a semi-ample $\Rr$-divisor $D$, such that $K_X+B+tM\equiv D$.
\end{conjecture}

\subsection{Length of extremal rays}\label{subsection: length of extremal rays and its applications}

The bound on the length of extremal rays also holds for g-pairs. Following \cite[Definition 1]{Sho09}, we define extremal curves.

\begin{definition}[Extremal curve]\label{def: extremal curve}
	An irreducible curve $C$ on $X/Z$ is called \emph{extremal} if it generates an extremal ray $R = \Rr_+[C]$ of the Kleiman-Mori cone $\overline{\rm NE}(X/Z)$, and has the minimal degree for this ray (with respect to any ample divisor).
\end{definition}

By definition, there exists an ample$/Z$ divisor $H$, such that 
\[
H \cdot C = \min\{H\cdot\Gamma \mid [\Gamma]\in R\}.
\] 

We have the following result on the length of extremal rays. We thank Chen Jiang for showing us the following simple proof. 
\begin{proposition}[The length of extremal rays for g-pairs]\label{prop: length of extremal rays for g-pair}
	Let $X$ be a $\Qq$-factorial klt variety, and $(X/Z,B+M)$ be a g-lc pair. Then for any $(K_X+B+M)$-negative extremal ray $R/Z$, there exists a curve $C$, such that $[C]\in R$, and $$0<-(K_X+B+M)\cdot C \leq 2\dim X.$$
\end{proposition}
 
\begin{proof}
	Let $C$ be any extremal curve such that $[C]\in R$. By definition, there exists an ample$/Z$ divisor $H$, such that 
	\[
	H \cdot C = \min\{H\cdot\Gamma \mid [\Gamma]\in R\}.
	\]
	By Lemma \ref{lem:glcklt}, for any $1 \gg \epsilon>0$, there exists a klt pair $(X, \Delta_\epsilon)$ with $K_X+\Delta_\epsilon \sim_\Rr K_X+B+M+\epsilon H$, such that $R$ is also a $(K_X+\Delta_\epsilon)$-negative extremal ray. By Kawamata's length of extremal rays, \cite{Kawamata91}, there exists a rational curve $\Gamma_\epsilon$, such that $[\Gamma_\epsilon]\in R$, and
	$$0<-(K_X+B+M+\epsilon H) \cdot \Gamma_\epsilon \leq 2\dim X.$$ 
	By the definition of extremal curve, we have $\frac{H \cdot C}{H \cdot \Gamma_{\epsilon}} \leq 1$, thus
	\begin{align*}
	-(K_X+B+M+\epsilon H) \cdot C &= -((K_X+B+M+\epsilon H) \cdot \Gamma_\epsilon )(\frac{H \cdot C}{H \cdot \Gamma_{\epsilon}})\\
	& \leq 2\dim X.
	\end{align*}
	Let $\epsilon \to 0$, we finish the proof.
\end{proof}

\subsection{Shokurov's polytope}\label{subsection: Shokurov polytope} In this subsection, we have the following setups. Let $X$ be a $\Qq$-factorial klt variety, $(X/Z, B + M)$ be an NQC g-lc polarized pair with data $\tilde X \xrightarrow{f} X \to Z$ and $\tilde M$. Suppose $\tilde M=\sum \mu_j \tilde{M_j}$, where $\tilde{M_i}$ is a nef $\Qq$-Cartier divisor and $\mu_i\ge 0$. Let $M_j$ be the pushforward of $\tilde{M_j}$ on $X$. Let $B=\sum b_iB_i$ be the prime decomposition.  

Consider the vector space 
\[
V \coloneqq (\oplus_i \Rr B_i) \oplus (\oplus_j \Rr M_j).
\] If $x_i, y_j \geq 0$, then $(X/Z, \sum_i x_i B_i + \sum_j y_j M_j)$ is a g-pair with data $\tilde X \to X$ and a nef divisor $\sum_j y_j \tilde M_j$. Let $\Delta=\sum d_iB_i$, $N=\sum \nu_j M_j$, and set the Euclidean norm
\[
||B+M-\Delta-N||=(\sum_i (b_i-d_i)^2+\sum_j (\mu_j-\nu_j)^2)^{1/2}.
\] 

The set of NQC g-lc pairs
	\[
	\Ll(B,M)\coloneqq \{\sum_i x_i B_i + \sum_j y_j M_j \in V \mid (X, \sum_{i} x_i B_i + \sum_j y_j M_j) \text{~is g-lc~}\}
	\] 
	is a rational polytope (may be unbounded). In fact, we may assume that $f:\tilde{X}\to X$ is a log resolution of $(X/Z,B+M)$. Given any $\Delta+N\in V$ with $\Delta \in (\oplus_i \Rr_{\ge0} B_i)$ and $N \in (\oplus_j \Rr_{\ge0} M_j)$, if we write 
	\[
	K_{\tilde{X}}+\tilde{\Delta}+\tilde{N}=f^{*}(K_X+\Delta+N),
	\]
	then the coefficients of $\tilde{\Delta}$ are \emph{rational} affine linear functions of the coefficients of $\Delta$ and $N$. The g-lc condition is the same as that such coefficients are chosen from $[0,1]$. Hence these rational affine linear functions cut out a rational polytope.

\begin{lemma}\label{lem:lengthbdd} Under the above notation. Let $X$ be a $\Qq$-factorial klt variety, and $(X/Z,B+M)$ be an NQC g-lc pair. 
	
	(1) Then there exists a positive real number $\alpha>0$, such that for any extremal curve $\Gamma$$/Z$, if $(K_X+B+M)\cdot \Gamma>0$, then $(K_X+B+M)\cdot \Gamma>\alpha$.
	
	(2) There exists a positive number $\delta>0$, such that if $\Delta+N\in \Ll(B,M)$, $||(\Delta-B)+(N-M)||<\delta$, and $(K_X+\Delta+N)\cdot R\le 0$ for an extremal ray $R/Z$, then $(K_X+B+M)\cdot R\le 0$. 
\end{lemma}
\begin{proof}(1) Because $\Ll(B, M)$ is a rational polytope, there exist $r,m \in \Nn, a_j \in \Rr_{\geq 0}$ and $D_j \in (\oplus_i \Qq_{\ge0} B_i), N_j \in (\oplus_j \Qq_{\ge0} M_j)$ such that $\sum_{j=1}^r a_j=1$, 
	\[
	(K_X+B+M)=\sum_{j=1}^r a_j(K_X+D_j+N_j),
	\]
	$(X/Z,D_j+N_j)$ is NQC g-lc, and $m(K_X+D_j+N_j)$ is Cartier. By Proposition \ref{prop: length of extremal rays for g-pair} and the definition of the extremal curve (see Definition \ref{def: extremal curve}), if $(K_X+D_j+N_j)\cdot \Gamma<0$, then $(K_X+D_j+N_j)\cdot \Gamma\ge -2\dim X$. Hence, if $(K_X+D_j+N_j)\cdot \Gamma\le 1$, then there are only finitely many possibilities for the intersection numbers $(K_X+D_j+N_j)\cdot \Gamma$. Therefore there are also finite many intersection numbers for $(K_X+B+M) \cdot \Gamma <1$, and the existence of $\alpha$ is clear.
	
	(2) Let $\mathbb B(B+M, 1) \subset V$ be a ball centered at $B+M$ with radius $1$. Because $\Ll(B,M)$ is a rational polytope which may be unbounded, we choose a \emph{bounded} rational polytope $\Ll'(B,M) \supset \Ll(B,M) \cap \mathbb B(B+M, 1)$. First, choose $\delta<1$. Let $\Delta'+N'$ be the intersection point of the line connecting $B+M$ and $\Delta+N$ on the boundary of $\Ll'(B,M)$, such that $\Delta+N$ lies inside the interval between $\Delta'+N'$ and $B+M$ (if this line lies on the boundary of $\Ll'(B,M)$, we choose $\Delta'+N'$ to be the furthest intersection point). There exist $r,s \in \Rr_{\geq 0}$, such that $r+s=1$, and $\Delta+N=r(B+M)+s(\Delta'+N')$. Suppose that there is an extremal ray $R/Z$ such that $(K_X+\Delta+N)\cdot R\le 0$ and $(K_X+B+M)\cdot R> 0$. Let $\Gamma$ be an extremal curve of $R$. By (1) there exists $\alpha>0$, such that $(K_X+B+M)\cdot \Gamma>\alpha$. If 
	\[
	r>\frac{2s\dim X}{\alpha}, \text{~or equivalently~} r>\frac{2\dim X}{2\dim X+\alpha},
	\] then by the proof of Proposition \ref{prop: length of extremal rays for g-pair}, we have
	\begin{align*}
	(K_X+\Delta+N)\cdot \Gamma=&r(K_X+B+M)\cdot \Gamma +s(K_X+\Delta'+N')\cdot\Gamma\\
	>&r\alpha-2s\dim X>0,
	\end{align*} which is a contradiction.
	
From the above discussion, we can construct the desired $\delta$ as follows. Let $l>0$ be the minimal non-zero distance from $B+M$ to the boundary of $\Ll'(B,M)$. We choose $0<\delta \ll1$ such that
	\[
	\frac{l-\delta}{l}>\frac{2\dim X}{2\dim X+\alpha}.
	\] Then by the choice of $l$, we see that $r=\frac{l'-\delta}{l'}\geq \frac{l-\delta}{l}$, where $l'$ is the distance from $B+M$ to $\Delta'+N'$. This $\delta$ satisfies the requirement.
\end{proof}
\begin{example}\label{rem:reasonNQC}
	Without the NQC assumption, Lemma \ref{lem:lengthbdd}(1) no longer holds true. Indeed, let $E$ be a general elliptic curve, $X=E\times E$. Fix a point $P\in E$, consider the divisor classes $f_1=[\{P\}\times E],f_2=[E\times\{P\}],\delta=[\Delta]$, where $\Delta\subset E\times E$ is the diagonal. According to \cite[Lemma 1.5.4]{LazarsfeldPositivity1}, $N=f_1+\sqrt{2}f_2+(\sqrt{2}-2)\delta$ is nef. It is not hard to show that for any $\epsilon>0$, there exists a curve $C$, such that $N\cdot C<\epsilon$ (see \cite{Rnefmathoverflow0508}).
\end{example}

Let $\dim (\oplus_i \Rr B_i) \oplus (\oplus_j \Rr M_j)=d$. For $k \in \Qq$, let
\[
	[0, k]^d \coloneqq [0, k] \times \cdots \times [0,k] \subset \Rr^d.\] If $\{R_t\}_{t\in T}$ is a family of extremal rays of $\overline{\rm NE}(X/Z)$, set 
\[
	\nN_{T}=\{\Delta+N\in \Ll(B,M) \mid (K_X+\Delta+N)\cdot R_t\ge0~\forall~ t\in T\}.
	\] 
	
By the same argument as \cite{Birkar11} (cf. the original proof in \cite{Sho09}), we have the following result for NQC g-lc pairs.

\begin{proposition}\label{le: decomposition to nef Cartier divisors} Under the above notation. Let $X$ be a $\Qq$-factorial klt variety, and $(X/Z,B+M)$ be an NQC g-lc pair. Then the set 
	\[
	\nN_{T} \cap [0, k]^d
	\]
	is a rational polytope for any $k\in\Qq$. In particular, if $K_X+B+M$ is nef$/Z$, then there exist NQC g-lc pairs $(X, D_i+N_i)$ with nef part $N_i$ and boundary part $D_i$, and $a_i \in \Rr_{>0}$,  such that 
	\[
	K_X+B+M=\sum_i a_i (K_X+D_i+N_i)\] with $\sum_{i} a_i=1$. Moreover, there exists $m\in \Nn$, such that $m(K_X+D_i+N_i)$ is a nef$/Z$ and Cartier divisor for each $i$.
\end{proposition}
\begin{proof}
 By definition, $\nN_{T} \cap [0, k]^d$ is just 
 \[
\{\Delta+N\in \Ll(B,M) \cap [0, k]^d \mid (K_X+\Delta+N)\cdot R_t\ge0~\forall~ t\in T\},
 \] and $\Ll(B,M) \cap [0, k]^d$ is a bounded rational polytope. 
	
	Since $\nN_{T} \cap [0, k]^d$ is compact, by Lemma \ref{lem:lengthbdd}(2), there are 
	\[
	(\Delta_1+N_1),\ldots,(\Delta_n+N_n)\in \nN_{T}\cap [0, k]^d,
	\] and $\delta_1,\ldots,\delta_n>0$, such that $\nN_{T} \cap [0, k]^d$ is covered by 
	\[
	\Bb_i=\{\Delta+N \in \nN_{T} \cap [0, k]^d \mid ||\Delta+N-(\Delta_i+N_i)||<\delta_i\}, 1 \leq i \leq n.
	\] Moreover, if $\Delta+N\in\Bb_i$ with $(K_X+\Delta+N)\cdot R_t<0$ for some $t\in T$, then  $(K_X+\Delta_i+N_i)\cdot R_t=0$. Let
	\[
	T_i=\{t\in T \mid (K_X+\Delta+N)\cdot R_t<0 \text{~for some~}\Delta+N\in \Bb_i\}.
	\]
	Then, $(K_X+\Delta_i+N_i)\cdot R_t=0$ for each $t\in T_i$. Moreover, we have 
	\[
	\nN_T \cap [0, k]^d=\bigcap_{i=1}^n (\nN_{T_i}\cap [0, k]^d).
	\] It suffices to show that $\nN_{T_i}\cap [0, k]^d$ is a rational polytope. 
	
	By replacing $T$ with $T_i$ and $\Delta+N$ with $\Delta_i+N_i$, we may assume that there exists $\Delta+N$ such that $(K_X+\Delta+N)\cdot R_t=0$ for any $t\in T$. Because 
	\[
	\{\Delta+N \in \Ll(B,M) \mid (K_X+\Delta+N) \cdot R_t = 0 ~\forall~t\in T\}
	\] is a rational polytope, we can further assume that $\Delta+N$ is a rational point. 
	
	We do induction on dimensions of polytopes. If $\dim (\Ll(B, M)\cap [0, k]^d)=1$, then the statement is straightforward to verify. If $\dim(\Ll(B, M)\cap [0, k]^d)>1$, let $\Ll^1,\ldots,\Ll^p$ be the proper faces of $\Ll(B, M)\cap [0, k]^d$. By induction on dimensions, $\nN_{T}^i \coloneqq \nN_{T}\cap \Ll^i$ is a rational polytope. If $\Delta' + N' \in \nN_T\cap [0, k]^d$, then there exists a line connecting with  $\Delta' + N'$ and $\Delta+ N$ which intersects some $\Ll^i$ on $\Delta'' + N''$. Moreover, we can assume that $\Delta' + N'$ lies inside the line segment between $\Delta+N$ and $\Delta'' + N''$.  Because $(K_X+\Delta+N)\cdot R_t=0$, we have $(K_{X''}+\Delta''+N'')\cdot R_t \geq 0$ for any $t\in T$. Thus $\Delta''+N'' \in \nN_{T}^i$. This shows that $\nN_T\cap [0, k]^d$ is the convex hull of $\Delta+N$ and all $\nN_T^i$, which is also a rational polytope.
\end{proof}

Let $D$ be a divisor on $X$, we say that a divisorial/flipping contraction $f: X \to Y$ is \emph{$D$-trivial}, if for any contraction curve $C$, we have $D \cdot C =0$.

\begin{lemma}\label{prop: K trivial}
	Let $(X/Z, (B+A)+M)$ be a $\Qq$-factorial NQC g-lc pair with boundary part $B+A$  and nef part $M$. Suppose that $X$ is klt, $(X/Z,B+M)$ is g-lc and $K_X+B+M$ is nef. Then there exists $\delta_0>0$, such that for any $\delta \in (0, \delta_0)$, any sequence of g-MMP$/Z$ on $(K_X+B+\delta A+M)$ is $(K_X+B+M)$-trivial.
\end{lemma}

\begin{proof}
	By Proposition \ref{le: decomposition to nef Cartier divisors}, there exist NQC g-lc pairs $(K_X+B'_k +M_k')$, such that 
	$K_X+B+M = \sum a_k (K_X+B'_k +M_k')$ with $\sum a_k =1, a_k >0$. Moreover, $K_X+B'_k +M_k'$ is a nef$/Z$ $\Qq$-Cartier divisor for each $k$. Thus there exists $m\in\Nn$, such that $m(K_X+B'_k +M_k')$ is Cartier. 
	
	Let 
	\[
	\alpha\coloneqq \min_k\{\frac{a_k}{m}\} \text{~and~} \delta_0=\frac{\alpha}{2\dim X+\alpha}.
	\] Choose $\delta \in (0,\delta_0)$, then for any extremal curve $C$ of $(K_X+B+M+\delta A)$, if $(K_X+B+M) \cdot C>0$, we have  $(K_X+B+M) \cdot C \geq \alpha$. By the length of extremal rays (Proposition \ref{prop: length of extremal rays for g-pair}),
	\begin{align*}
	&(K_X+B+M+\delta A)\cdot C\\
	=&\delta(K_X+B+M+A)\cdot C+(1-\delta)(K_X+B+M)\cdot C\\
	\ge& -2\delta\dim X+(1-\delta)\alpha>0.
	\end{align*} This is a contradiction. 
	Thus, any $(K_X+B+M+\delta A)$-flip or divisorial contraction, $f: X \dashrightarrow Y/Z$, is $(K_X+B+M)$-trivial. As $K_X+B'_k +M_k'$ is nef, $f$ is also $(K_X+B'_k +M_k')$-trivial, and thus $m(K_Y+B'_{Y,k}+M'_{Y,k})\coloneqq mf_{*}(K_X+B'_k +M_k')$ is nef and Cartier. We can repeat the above argument on $Y$. This proves the claim. 
\end{proof}

The dual of the above result is the following lemma. 

\begin{lemma}\label{prop: P trivial}
	Let $(X/Z, B+M)$ be a g-lc pair with $X$ a $\Qq$-factorial klt variety. Suppose that $P$ is an NQC divisor$/Z$. Then for any $\beta \gg 1$, any sequence of g-MMP$/Z$ on $(K_X+B+M+\beta P)$ is $P$-trivial.
\end{lemma}
\begin{proof}
	Since $P$ is NQC, there exists $\alpha>0$, such that for any curve $C/Z$, if $P\cdot C\neq 0$, then $P\cdot C>\alpha$. Set $d=\dim X$ and choose $\beta>\frac{2d}{\alpha}$. Suppose that $C$ is an extremal curve such that $(K_X+B+M+\beta P)\cdot C<0$. If $P\cdot C\neq0$, then by the length of extremal rays (Proposition \ref{prop: length of extremal rays for g-pair}), we have
	\begin{align*}
	(K_X+B+M+\beta P)\cdot C&=(K_X+B+M)\cdot C+\beta P\cdot C\\
	&\ge -2d+\beta\cdot\alpha>0.
	\end{align*}
	This is a contradiction, and thus $P\cdot C=0$.  Just as Lemma \ref{prop: K trivial}, by the $P$-triviality, we can continue this process and $\alpha$ is independent of this g-MMP. 
\end{proof}

\subsection{G-MMP with scaling of an NQC divisor}  \label{subsection: scalingnqcdiv}
In this subsection, we will define a g-MMP with scaling of a divisor $Q=E+P$, where $E$ is an effective divisor and $P$ is the pushforward of an NQC divisor. Notice that $P$ may not be an effective divisor. To emphasize this special property, we coin the name ``g-MMP with scaling of an NQC divisor'', although there also exists an effective part (i.e. $E$) in $Q$.

\begin{lemma}\label{lem:MMPscalingNQC}
	Let $X$ be a $\Qq$-factorial klt variety, and $(X/Z,B+M)$ be an NQC g-lc pair with boundary part $B$ and nef part $M$. Let $E$ be an effective divisor on $X$ and $P$ be a pushforward of an NQC divisor from a sufficiently high model. Set $Q=E+P$. Suppose that $(X/Z,(B+E)+(M+P))$ is NQC g-lc with boundary part $B+E$ and nef part $M+P$, and $K_X+(B+E)+(M+P)$ is nef$/Z$. Then, either $K_X+B+M$ is nef$/Z$, or there is an extremal ray $R/Z$ such that $(K_X+B+M)\cdot R<0$, $(K_X+B+M+\nu Q)\cdot R=0$, where
	\[
	\nu:=\inf\{t\ge0| K_X+B+M+tQ \text{~ is nef}/Z\}.
	\] In particular, $K_X+B+M+\nu Q$ is nef$/Z$, 
\end{lemma}
\begin{proof}
	Suppose that $K_X+B+M$ is not nef$/Z$. Let $\{R_i\}_{i\in\Ii}$ be the set of $(K_X+B+M)$-negative extremal rays$/Z$, and $\Gamma_i$ be an extremal curve of $R_i$. 
	As $\Ll(B, M)$ is a rational polytope, by Proposition \ref{prop: length of extremal rays for g-pair}, there are $r_1,\ldots,r_s \in \Rr_{>0}$ and $m \in \Nn$, such that
	\[-2\dim X\le (K_X+B+M)\cdot\Gamma_i=\sum_{j=1}^s\frac{r_j n_{i,j}}{m}<0,\] where $-2m(\dim X)\le n_{i,j}\in\Zz$. By Proposition \ref{le: decomposition to nef Cartier divisors}, there are $r'_1,\ldots,r'_t \in \Rr_{>0}$ and $m \in \Nn$ (after changing the above $m$ by a sufficiently divisible multiple), such that
	\[(K_X+B+E+M+P)\cdot\Gamma_i=\sum_{k=1}^t \frac{r'_kn'_{i,k}}{m},\] where $n'_{i,k}\in\Zz_{\geq 0}$. 
	
	Since $n_{i,j}$ is bounded above, $\{n_{i,j}\}$ is a finite set, and so is $\{\sum_{j}r_j n_{i,j}\}$. Moreover, $\sum_{k} r'_kn'_{i,k}$ belongs to a DCC set, where DCC stands for descending chain condition.
	Let 
	\[\nu_i\coloneqq\frac{-(K_X+B+M)\cdot \Gamma_i}{Q\cdot \Gamma_i}.\]
	Thus, 
	\[\frac{1}{\nu_i}=\frac{\sum_{k} r'_kn'_{i,k}}{-\sum_{j}r_j n_{i,j}}+1\]
	belongs to a DCC set. Hence there exists a maximal element $\nu=\nu_s$ in the set $\{\nu_i\}_{i\in\Ii}$. Then,
	\[(K_X+B+M+\nu Q)\cdot \Gamma_i\ge0\]
	for any $i\in\Ii$. For the extremal curve $\Gamma_s$,  $(K_X+B+M+\nu Q)\cdot \Gamma_s=0$.	 
\end{proof}
\begin{definition}[G-MMP with scaling of an NQC divisor]\label{defn:MMPsP} Under the same assumptions and notation of Lemma \ref{lem:MMPscalingNQC}, we define the g-MMP with scaling of an NQC divisor as follows.

(1) If $K_X+B+M$ is nef$/Z$, we stop.

(2) If $K_X+B+M$ is not nef$/Z$, there exists an extremal ray $R$ as in Lemma \ref{lem:MMPscalingNQC}. By Lemma \ref{lem:glcklt}, we can contract $R$. 
\begin{itemize}
\item If the contraction is a Mori fiber space, we stop. 
\item If the contraction is a divisorial (resp. flipping) contraction, let  $X\dashrightarrow X_1$ be the corresponding contraction (resp. flip). Let $(X_1,B_1+M_1+\nu Q_1)$ be the birational transform of $(X, B+M+\nu Q)$. We can continue the previous process on $(X_1,B_1+M_1+\nu Q_1)$ in place of $(X, B+M+Q)$. In fact, by $(X, B+M+\nu Q)$-triviality,  $(X_1,B_1+M_1+\nu Q_1)$ is nef$/Z$.\end{itemize}

By doing this, we obtain a sequence (may be infinite) of varieties $X_i$ and corresponding nef thresholds $\nu_i\coloneqq \inf\{t\ge0| K_{X_i}+B_i+M_i+tQ_i \text{~ is nef}/Z\}.$
\end{definition}

\begin{remark}\label{rk: decreasing sequence}
By definition, the nef thresholds $\nu_i \geq \nu_{i+1}$ for each $i$.
\end{remark}

\subsection{Lifting the sequence of flips}
\label{subsection: Lifting a sequence of flips with quasi-scaling}
We use the same notation as Section \ref{subsection: scalingnqcdiv}. Suppose that a sequence of g-MMP$/Z$ with scaling of $Q=E+P$ on $K_{X}+B+M$ only consists of flips, $X_i\dashrightarrow X_{i+1}/Z_i$, where $X_i \to Z_i$ is the flipping contraction.
Let $h_0:(X'_0/Z,B'_0+M'_0)\to X_0$ be a g-dlt modification of $(X_0,B_0+M_0)$ (see Proposition \ref{prop: dlt}). Pick an ample$/Z_0$ divisor $H\ge0$, such that 
$$K_{X_0}+B_0+M_0+H\sim_{\Rr}0/Z_0,$$
and $(X_0,B_0+M_0+H)$ is g-lc. 
By Lemma \ref{lem:glcklt}, $(X_0,\Delta_0)$ is klt for some boundary $\Delta_0\sim_{\Rr} B_0+M_{0}+\epsilon H/Z_0$. According to the proof of Lemma \ref{lem:glcklt}, we may choose $\Delta_0\ge0$, such that $h_0^{*}(K_{X_0}+\Delta_0)=K_{X'_0}+\Delta'_0$
for some effective divisor $\Delta'_0$, and $({X'_0},\Delta'_0)$ is klt. Now run an MMP$/Z_0$ on $K_{X'_0}+\Delta'_0$ with scaling of $h_0^{*}(H)$. By \cite[Corollary 1.4.2]{BCHM10}, the MMP terminates with a log terminal model, $X'_0\dashrightarrow X'_1$. By construction, we have
\begin{align*}
(1-\epsilon)(K_{X'_0}+B'_{0}+M'_{0})&\sim_{\Rr}
(1-\epsilon)h_0^{*}(K_{X_0}+B_{0}+M_{0})\\
&\sim_{\Rr}h_0^{*}(K_{X_0}+B_{0}+M_{0}+\epsilon H)/Z_0\\
&\sim_{\Rr}K_{X'_0}+\Delta'_0/Z_0.
\end{align*}  
Thus this MMP is also a g-MMP$/Z_0$ on $K_{X'_0}+B'_{0}+M'_{0}$, and thus $K_{X'_1}+B'_{1}+M'_{1}$ is nef$/Z_0$. We define $Q_0'=h_0^{*}(Q_0)$ as follows. Suppose that 
\[
W \xrightarrow{p} X_0' \xrightarrow{h_0} X_0
\] is a sufficiently high log resolution such that $P_0=(h_0 \circ p)_* P_W$ for an NQC divisor $P_W$ on $W$. Then by the negativity lemma, $P_W+F=(h_0 \circ p)^*P_0$ with $F \geq 0$. Set 
\begin{equation}\label{eq: E', P'}
E_0'\coloneqq h_0^*E_0+p_*F \text{~and~} P_0' \coloneqq p_*P_W,
\end{equation} and 
\begin{equation}\label{eq: Q'}
Q_0'\coloneqq E_0'+P_0'=h_0^*(E_0)+p_*(p^*\circ h_0^*(P_0))=h_0^*(E_0+P_0)=h_0^*(Q_0).
\end{equation}

 Because $\rho(X_0/Z_0)=1$, $Q_0\equiv aH/Z_0$ for some $a>0$. Thus the g-MMP/$Z_0$ $X'_i \dashrightarrow X'_{i+1}$ is also a g-MMP$/Z_0$ on $K_{X'_0}+B'_{0}+M'_{0}$ with scaling of $Q_0'$.

Because $X_0, X_1$ are isomorphic in codimension $1$ and $K_{X_1}+B_{1}+M_{1}$ is ample$/Z_0$, $({X_1}, B_{1}+M_{1})$  is a g-log canonical model of $({X_0}, B_0+M_0)$ over $Z_0$ (here g-log canonical model means a g-log terminal model with $K_{X_1}+B_{1}+M_{1}$ ample$/Z_0$). Thus there exists a morphism $h_1: X'_1\to X_1$ such that  $K_{X'_1}+B'_{1}+M'_{1}=h_1^*(K_{X_1}+B_{1}+M_{1})$, which is also a g-dlt modification of $({X_1}, B_{1}+M_{1})$. We can continue the above process for $X_1, X'_1$, etc. in places of $X_0, X_0'$, etc. 

From the above, we have a sequence of g-MMP$/Z$ on $(K_{X'_0}+B'_{0}+M'_{0})$ with scaling of $Q_0'$. The reason is as follows. A priori, the g-MMP $X'_i \dashrightarrow X'_{i+1}$ with scaling of $Q_0'$ is over $Z_0$ \emph{rather than over $Z$}. We denote this g-MMP$/Z_0$ by 
\[
X_i'=Y_0 \dashrightarrow Y_1 \dashrightarrow\cdots \dashrightarrow Y_{k}=X_{i+1}',
\] and let $\nu_j', 0 \leq j \leq k$ be the corresponding nef thresholds$/Z_0$. By $K_{X_i'}+B_i'+M_i'+\nu_iQ_i' \equiv 0/Z_0$, we have $\nu_j'=\nu_i$ for $0 \leq j \leq k-1$. Thus
\[
\nu_j' = \inf\{t \mid K_{Y_j}+B_{Y_j}+M_{Y_j}+tQ_{Y_{j}} \text{~is nef over~} Z\},
\] for $0 \leq j \leq k-1$. This shows that the g-MMP with scaling of $Q_i'$ is also over $Z$.

By doing above, we lift the original g-MMP with scaling to a new g-MMP with scaling. The advantage is that  each $(X'_i, B_i'+M'_i)$ becomes $\Qq$-factorial and g-dlt.

\section{Special termination for g-MMP with scaling}\label{sec: special termination}

It is crucial to observe that some termination results still hold for g-MMP with scaling of an NQC divisor. The following is a variation of \cite[Theorem 1.9]{Birkar12}.
\begin{theorem}\label{thm: gmmtermination}
Under the assumptions and notation of Definition \ref{defn:MMPsP}. Suppose that there is a g-MMP with scaling of $Q$. Let $\mu = \lim_{j \to \infty} \nu_j$. If $\mu\neq \nu_j$ for any $j$, and $(X/Z,(B+\mu E)+(M+\mu P))$ has a g-log minimal model, then the g-MMP terminates. 
\end{theorem}

This theorem is proved in several steps.

\begin{proposition}\label{prop: a special termination 1}
	Under the above notation, Theorem \ref{thm: gmmtermination} holds if there is a birational map $\phi:X\dashrightarrow Y/Z$ between $\Qq$-factorial varieties satisfying:
	\begin{enumerate}
		\item the induced map $X_i\dashrightarrow Y$ is isomorphic in codimension one for every $i$, 	
		\item $(Y/Z,(B_Y+\mu E_Y)+(M_Y+\mu P_Y))$ is a g-log minimal model of $(X/Z,(B+\mu E)+(M+\mu P))$,	
		\item there is a reduced divisor $A \geq 0$ on $X$, whose components are movable divisors and generate $N^1(X/Z)$,  	
		\item there exists $\epsilon>0$, such that $(X/Z,(B+E+\epsilon A)+(M+P))$ is g-dlt with boundary part $(B+E + \epsilon A)$ and nef part $(M+P)$,
		\item  there exists $\delta>0$, such that $(Y/Z, (B_Y+(\mu+\delta)E_Y + \epsilon A_Y) + (M_Y+(\mu+\delta)P_Y))$ is g-dlt with boundary part $(B_Y+(\mu+\delta) E_Y + \epsilon A_Y)$ and nef part $(M_Y+(\mu+\delta) P_Y))$.
	\end{enumerate}
\end{proposition}

\begin{proof}
	\noindent Suppose that the g-MMP does not terminate. Pick $j\gg 1$, so that $\nu_{j-1} >\nu_{j}$. Then $X\dashrightarrow X_j$ is a partial g-MMP/$Z$ on $K_{X}+B+M+\nu_{j-1}Q$. It is also a partial g-MMP/$Z$ on $K_{X}+B+M+\nu_{j-1}Q+\epsilon A$ after replacing $\epsilon$ with a smaller number. In particular, $(X_j/Z,(B_j+\epsilon A_j)+M_j+\nu_{j-1}Q)$ is $\Qq$-factorial g-dlt, where $A_j$ is the birational transform of $A$ on $X_j$. As $j \gg1$, after reindexing, we may assume that the g-MMP only consists of flips$/Z$ starting with $(X_1/Z,B_1+M_1)=(X/Z,B+M)$. Moreover, by replacing $B+M$ with $B+M+\mu Q$, we may assume that $\mu=0$.
		
	Possibly by choosing a smaller $\epsilon$ again, by Lemma \ref{prop: K trivial}, we may assume that any sequence of g-MMP$/Z$ on $(K_{X_{j}}+(B_j+\nu_{j-1} E_j + \epsilon A_j)+(M_j+\nu_{j-1} P_j))$  is a sequence of $(K_{X_j}+(B_j+\nu_{j-1}E_j)+(M_j+\nu_{j-1} P_j))$-flop. By assumption, $K_{X_j}+(B_j+\nu_{j-1} E_j +\epsilon A_j)+(M_j+\nu_{j-1}P_j)$ is a limit of movable$/Z$ $\Rr$-divisors. Since the components of $A_j$ generate $N^1(X_j/Z)$, there exists a general ample$/Z$ divisor $H$ and an effective divisor $H' < A_j$, such that $A_j \sim_\Rr H+H'$, and $(X_j/Z, (B_j+\nu_{j-1} E_j+\epsilon H' + \epsilon H)+(M_j+\nu_{j-1} P_j))$ is g-dlt. By Lemma \ref{lem:glcklt}, there exists a klt pair $(X_j, \Delta_j)$ such that
	\[
	 K_{X_j}+\Delta_j \sim_\Rr K_{X_j}+(B_j+\nu_{j-1} E_j + \epsilon H' + \epsilon H)+(M_j+\nu_{j-1} P_j).
	\] 
	By \cite{BCHM10}, we may run an MMP$/Z$ with scaling of an ample divisor on $K_{X_j}+\Delta_j$, which is the same as an MMP$/Z$ on $(K_{X_j}+(B_j+\nu_{j-1} E_j + \epsilon A_j)+(M_j+\nu_{j-1} P_j))$. It terminates with a g-log minimal model $(T/Z, (B_T+\nu_{j-1} E_T+\epsilon A_T)+ (M_T+ \nu_{j-1} P_T))$. Notice that $X_j, T$ are isomorphic in codimension $1$, and $(K_T+(B_T+\nu_{j-1} E_T)+ (M_T+ \nu_{j-1} P_T))$ is nef$/Z$. Again, since the components of $A_T$ generate $N^1(T/Z)$, we can choose $0 < D_T \leq A_T$ such that $K_{T}+B_T+M_T+\nu_{j-1}Q_T+\epsilon D_T$ is ample. Moreover, $K_T+B_T+M_T+\nu_{j-1} Q_T$ is nef$/Z$ by the choice of $\epsilon$.

For the same reason, possibly by choosing smaller $\nu_j$ and $\epsilon$, we can run a g-MMP$/Z$ on $(K_Y+B_Y+M_Y+\nu_{j-1}Q_Y+\epsilon D_Y)$ with scaling of an ample divisor, and get a g-log minimal model, $(Y', B_{Y'}+M_{Y'}+\nu_{j-1} Q_{Y'}+\epsilon D_{Y'})$, such that both $K_{Y'}+B_{Y'}+M_{Y'}+\nu_{j-1} Q_{Y'}+\epsilon D_{Y'}$ and $K_{Y'}+B_{Y'}+M_{Y'}$ are nef (see Lemma \ref{prop: K trivial}). Because $Y, Y'$ are $\Qq$-factorial varieties which are isomorphic in codimension $1$ and $K_{T}+B_T+M_T+\nu_{j-1}Q_T+\epsilon D_T$ is ample$/Z$,  we have $Y'=T$. Hence, both 
	\[
	K_T+B_T+M_T+\nu_{j-1} Q_T \text{~and~} K_T+B_T+M_T
	\] are nef$/Z$. By $\nu_{j-1}>\nu_{j}>\mu=0$, $K_T+B_T+M_T+\nu_{j}Q_T$ is nef$/Z$.

Let $r: U \to X_j, s: U \to T$ be a common log resolution. By the negativity lemma, we have
	\begin{align*}
	r^*(K_{X_{j}}+B_{j}+M_{j}) > &s^*(K_T+B_T+M_T),\\
	r^*(K_{X_{j}}+B_{j}+M_{j}+\nu_{j-1} Q_{j}) =& s^*(K_T+B_T+M_T+\nu_{j-1} Q_T),\\
	r^*(K_{X_{j}}+B_{j}+M_{j}+\nu_{j} Q_{j}) =& s^*(K_T+B_T+M_T+\nu_{j} Q_T).
	\end{align*} 
	This is a contradiction.
\end{proof}
\begin{proof}[Proof of Theorem \ref{thm: gmmtermination}]
Let $(Y/Z,(B_Y+\mu E_Y)+(M_Y+\mu P_Y))$ be the g-log minimal model of $(X/Z,(B+\mu E)+(M+\mu P))$ with corresponding map $\phi: X \dashrightarrow Y$. As in Proposition \ref{prop: a special termination 1}, we may assume that $\mu=0$, and the g-MMP$/Z$ only consists of flips, $X_i\dashrightarrow X_{i+1}$. Because there are finite many g-lc centers, we can assume that no g-lc centers are contracted in the sequence. Moreover, choose $\nu_{i-1}>\nu_{i}$, then for any birational morphism $f: W \to X_i$, we can write 
\begin{equation}\label{eq: pullback Q}
f^*Q_i = f^*(E_i+P_i)=\tilde E_i+P_{W,i}+\Theta_{W,i},
\end{equation} with $\tilde E_i$ the birational transform of $E_i$. The meanings of $P_{W,i}, \Theta_{W,i}$ are as follows (cf. \eqref{eq: E', P'}). By definition, we can assume that $P_i=q_*P'$ where $q: W' \to X_i$ is a sufficiently high model and $P'$ is an NQC divisor. By taking a common log resolution, we can assume that there also exists a morphism $p: W' \to W$. Then we set $P_{W,i}=p_*P'$. By $E_i \geq 0$ and the negativity lemma,  $\Theta_{W,i} \geq 0$ is a $f$-exceptional divisor. By $(X_{i}, B_{i}+M_{i}+\nu_{i-1}Q_{i})$ is g-lc, there is no g-lc place of $(X_{i}, B_{i}+M_{i}+\nu_{i}Q_{i})$ which is contained in $\Supp \Theta_{W,i}$. We can replace $(X/Z,B+M)$ with $(X_i/Z,B_i+M_i)$ and $Q$ with $\nu_iQ_i$.

	Step 1. Let $f: W \to X$ and $g: W \to Y$ be a sufficiently high common log resolution of $(X/Z,(B+E)+(M+P))$ and $(Y/Z,B_Y+M_Y+Q_Y)$. We have
	\begin{equation}\label{eq: common resolution}
	\begin{split}
	F &\coloneqq f^*(K_{X}+B+M) -g^*(K_Y+B_Y+M_Y), \text{~and}\\
	F' &\coloneqq K_W+B_W+M_W - f^*(K_{X}+B+M), 
	\end{split}
	\end{equation}
	where $B_W$ is defined as \eqref{eq: B_Y}. Then $F, F'$ are effective exceptional divisors over $Y, X$ respectively. By the definition of g-log minimal model, $F'$ is also exceptional over $Y$.
	
	Let $E_W$ be the birational transform of $E$ on $W$, and $P_W$ be the nef$/Z$ divisor corresponding to $P$ on $W$. Set $Q_W=P_W+E_W$. We have
	\begin{equation}\label{eq: W}
	\begin{split}
	K_W+B_W+M_W &\equiv F+F'/Y.
	\end{split}
	\end{equation}  
	By Proposition \ref{prop: termination for very exceptional divisor}, we can run a g-MMP$/Y$ on $(K_W+B_W+M_W)$ with scaling of an ample divisor, and it terminates with a model $Y'$, such that $F+F'$ is contracted. Thus $(Y',B_{Y'}+M_{Y'})$ is a g-dlt modification of $(Y, B_Y+M_Y)$.

Step 2. We prove that $\phi: X\dashrightarrow Y$ does not contract any divisor. Otherwise, let $D$ be a prime divisor on $X$ which is contracted by $\phi$, and $D_W$ be the birational transform of $D$ on $W$. Since $a(D,X,B+M)<a(D,Y,B_Y+M_Y)$, $D_W$ is a component of $F$. In Step 1, the g-MMP$/Y$ on $(K_W+B_W+M_W)$ contracts $D_W$. We will get a contradiction as follows. Let $\nu_i$ be sufficiently small so that $W\dashrightarrow Y'$ is a partial g-MMP$/Z$ on $K_W+B_W+M_W+\nu_i Q_W$. Since $(X/Z,B+M+\nu_iQ)$ is g-lc, 
	$$K_W+B_W+M_W+\nu_i Q_W-f^{*}(K_X+B+M+\nu_i Q)$$
	is effective and exceptional over $X$. On the other hand, $X\dashrightarrow X_i$ is a partial g-MMP$/Z$ on $(K_X+B+M+\nu_i Q)$, we have $$f^{*}(K_X+B+M+\nu_i Q)\ge N,$$ 
	where 
	\[
	N=p_{*}q^{*}(K_{X_{i}}+B_i+M_i+\nu_i Q_i)
	\]
	for some common log resolution $p:W'\to W, q:W'\to X_i$. Since $K_{X_{i}}+B_i+M_i+\nu_i Q_i$ is nef$/Z$, $N$ is a pushforward of a nef divisor. In particular, $N$ is a limit of movable$/Z$ divisors. We have
	\[
	K_W+B_W+M_W+\nu_i Q_W=N+G,
	\]
	where $G \geq 0$ is exceptional over $X$. Here we use the fact that $X$ and $X_i$ are isomorphic in codimension one. Since $G$ is exceptional$/X$, $D_W$ is not a component of $G$. For the g-MMP in Step 1, if $D_W$ were contracted by an extremal contraction of a curve $C$, we have $(K_W+B_W+M_W+\nu_i Q_W) \cdot C<0$. But $N \cdot C \geq0$ and $G \cdot C\geq 0$. Thus $D_W$ cannot be contracted. This is a contradiction.
	
Step 3. From $W$, we construct a g-dlt modification of $(X, B+M)$. Let
	\[
	F'' \coloneqq K_W+(B_W+E_W)+(M_W+P_W) -f^*(K_{X}+(B+E)+(M+P)),
	\] which is effective and exceptional over $X$. We run a g-MMP$/X$ on $K_W+(B_W+E_W)+(M_W+P_W)$ which terminates with a model $h: X' \to X$ and contracts $F''$. This $h$ is a g-dlt modification of $(X, (B+E)+(M+P))$. Let
	\[
	K_{X'}+(B'+E')+(M'+P') =h^*(K_{X}+(B+E)+(M+P)),
	\]  
	where $E'$ is the strict transform of $E$ and $P'$ is the pushforward of $P_W$. By assumption (see the paragraph before Step 1) that for 
	\[
	Q'\coloneqq h^*(E+P)=E'+P'+\Theta'
	\] as in \eqref{eq: pullback Q}, there is no g-lc place of $({X}, B+M)$ which is contained in $\Theta'$. Thus $\Theta'=0$. Hence $h$ is also a g-dlt modification of $({X}, B+M)$, that is
	\[
	K_{X'}+B'+M' =h^*(K_{X}+B+M).
	\] In particular, $h$ extracts all the g-lc places of $(X, B+M)$ on $W$. Because $\phi^{-1}: Y \dashrightarrow X$ can only extract g-lc places of $(X,B+M)$ (see Remark \ref{rmk: extract lc places}), we see that these divisors are all on $X'$.
	
Step 4. By Subsection \ref{subsection: Lifting a sequence of flips with quasi-scaling}, we can lift the sequence $X_i\dashrightarrow X_{i+1}/Z_i$ to a g-MMP$/Z$ on $K_{X'}+B'+M'$ with scaling of $Q'$. Hence, each $({X_i'}, B_i'+M_i')$ is $\Qq$-factorial and g-dlt.

Step 5. Possibly by replacing $X'$ with $X'_i$ for $i\gg 1$, we show that $X', Y'$ are also isomorphic in codimension $1$, and $(Y'/Z,B'+M')$ is a g-log minimal model of $(X/Z,B+M)$.
	
	First, We show that $Y'\dashrightarrow X'$ does not contract any divisor. Suppose that $D\subset Y'$ is a prime divisor which is exceptional over $X'$. If $D$ is on $Y$, then $a(D,X,B+M)=0$ as $D$ is exceptional over $X$. Thus, by Step 3, $D$ is on $X'$, a contradiction. If $D$ is exceptional over $Y$, as $(Y',B_{Y'}+M_{Y'})$ is a g-dlt modification of $(Y, B_Y+M_Y)$, we have $a(D,Y,B_Y+M_Y)=0$. This implies that $a(D,X,B+M)=0$, and again we get a contradiction from Step 3. 
	
	Next, We show that $X' \dashrightarrow Y'$ does not contract any divisor. Possibly by replacing $X'$ with $X'_i$ for $i\gg 1$, we may assume that the g-MMP$/Z$ on $(K_{X'}+B'+M')$ with scaling of $Q'$ only consists of flips. By using the same method as Step 2, it suffices to show that $(Y'/Z,B_{Y'}+M_{Y'})$ is a g-log minimal model of $(X/Z,B+M)$. Thus we only need to compare g-log discrepancies. Suppose that $D\subset X'$ is a prime divisor which is exceptional over $Y'$. Since $X, Y$ are isomorphic in codimension $1$, $D$ is exceptional over $X$. Hence $a(D, X', B'+M')=a(D, X, B+M)=0$. If $a(D,Y',B_{Y'}+M_{Y'})=0$, then $a(D,Y,B_Y+M_Y)=0$. Thus the birational transform of $D$ cannot be a component of $F+F'$ in \eqref{eq: common resolution}, and it can not be contracted over $Y'$. This is a contradiction. Therefore, $a(D,Y',B_{Y'}+M_{Y'})>0$, which implies that 
	$(Y'/Z,B'+M')$ is a g-log minimal model of $(X/Z,B+M)$.

Step 6. Let $A\ge0$ be a reduced divisor on $W$ whose components are general ample$/Z$ divisors such that they generate $N^1(W/Z)$. Since $X'$ is obtained by running some g-MMP on $K_W+B_W+M_W+Q_W$, this g-MMP is also a partial g-MMP on $K_W+B_W+M_W+Q_W+\epsilon A$ for any $1 \gg \epsilon>0$. In particular, $(X'/Z,(B'+E'+\epsilon A')+(M'+P'))$ is g-dlt, where $A'$ is the birational transform of $A$. For similar reasons, we can choose $1\gg\epsilon,\delta>0$, so that $(Y'/Z,(B_{Y'}+\delta E_{Y'}+\epsilon A_{Y'})+M_{Y'}+\delta P_{Y'})$ is also g-dlt.

Now, by Proposition \ref{prop: a special termination 1}, the g-MMP$/Z$, $X'_i \dashrightarrow X'_{i+1}$, terminates. This implies that the original g-MMP$/Z$, $X_i \dashrightarrow X_{i+1}$, also terminates. This finishes the proof.
\end{proof}

We introduce the notion of difficulty for g-pairs before proving the special termination.

\begin{definition}[Difficulty for g-pairs]
	Let $(X,B+M)$ be a $\Qq$-factorial g-dlt pair with data $\tilde X \xrightarrow{f} X \to Z$ and $\tilde M$. For $\Rr$-divisors $B, \tilde M$, assume that $B=\sum b_j B_j$ is the prime decomposition of $B$, and $\tilde M=\sum \mu_i \tilde M_i$ with $\tilde M_i$ a nef$/Z$ Cartier divisor for each $i$.
	Let $\bm{b}=\{b_j\}$, $\bm{\mu}=\{\mu_i\}$. Recall that
	\[\mathbb{S}(\bm{b},\bm{\mu})= \{1-\frac{1}{m}+\sum_{j} \frac{r_jb_j}{m}+\sum_{i}\frac{s_i \mu_i}{m} \leq 1\mid m\in\mathbb{Z}_{>0},r_j,s_i\in\Zz_{\ge0}\}\cup\{1\}.\]
	Let $S$ be a g-lc center of $(X,B+M)$, then
	\[
	K_{S}+B_{S}+M_{S} = (K_{X}+B+M)|_{S}
	\] is defined in Proposition \ref{prop: intersection on g-lc centers}.
	The \emph{difficulty} of the g-pair $(X,B+M)$ is defined to be
	\begin{align*}
	&d_{\bm{b},\bm{\mu}}(S,B_S+M_S)\\
	=&\sum_{\alpha\in \mathbb{S}(\bm{b},\bm{\mu})}\#\{E \mid a(E,B_S+M_S)<1-\alpha,\Center_{S}(E)\nsubseteq \lfloor B_S \rfloor\}\\
	&+\sum_{\alpha\in\mathbb{S}(\bm{b},\bm{\mu})}\#\{E \mid a(E,B_S+M_S)\le1-\alpha, \Center_{S}(E)\nsubseteq \lfloor B_S \rfloor\}.
	\end{align*}
\end{definition}

\begin{remark}
Notice that $d_{\bm{b},\bm{\mu}}(S,B_S+M_S)$ is slightly different from \cite[Definition 4.2.9]{Fujino07} (cf. \cite[7.5.1 Definition]{Fli92}): in the second summand, we also includes $E$ whose g-log discrepancy \emph{equals} $1-\alpha$. By doing this, the standard argument can be simplified (cf. \cite[Proposition 4.2.14]{Fujino07} and the argument below). Just as for log pairs, $d_{\bm{b},\bm{\mu}}(S,B_S+M_S)<+\infty$ (cf. \cite[4.12.2 Lemma]{Fli92}). 
\end{remark}

\begin{theorem}\label{prop: special termination 2}
	Under the assumptions and notation of Definition \ref{defn:MMPsP}. We run a g-MMP$/Z$ with scaling of $Q$ on $K_X+B+M$. Assume the existence of g-log minimal models for pseudo-effective NQC g-lc pairs in dimensions $\le \dim X-1$. Suppose that $\nu_i>\mu$ for $\mu=\lim \nu_i$ (in particular, the g-MMP is an infinite sequence). Then, after finitely many steps, the flipping locus is disjoint from the birational transform of $\lf B \rf$. \end{theorem}

\begin{proof}
We follow the proof of \cite[Theorem 4.2.1]{Fujino07}. 
	
Because the number of g-lc centers of a g-lc pair is finite, we may assume that after finitely many steps, the flipping locus contains no g-lc centers. Thus $\phi_i: X_i\dashrightarrow X_{i+1}$ induces an isomorphism of $0$-dimensional g-lc centers for every $i$. 
	
	 We show that $\phi_i$ induces an isomorphism of every g-lc center by induction on dimensions $d$ of g-lc centers. Then the theorem follows from $d=\dim X -1$. Now, for each $d \leq k-1$, we assume that $\phi_i$ induces an isomorphism for every $d $-dimensional g-lc centers.
	
	Let $S$ be a $k$-dimensional g-lc center of $(X,B+M)$, and $S_i$ be the birational transform of $S$ on $X_i$. By adjunction formula (Proposition \ref{prop: intersection on g-lc centers}), $(K_{X_i}+B_i+M_i)|_{S_i}=K_{S_i}+B_{S_i}+M_{S_i}$, and the coefficients of $B_{S_i}$ belong to the set $\mathbb{S}(\bm{b},\bm{\mu})$. By the induction hypothesis, after finitely many flips, $\phi_i$ induces an isomorphism on $\lfloor B_{S_i}\rfloor$, and thus $\Center_{S_i}E\subseteq \lfloor B_{S_i}\rfloor$ if and only if $\Center_{S_{i+1}}E\subseteq \lfloor B_{S_{i+1}}\rfloor$. By the negativity lemma, $a(E,S_i,B_{S_i}+M_{S_i})\le a(E,S_{i+1},B_{S_{i+1}}+M_{S_i})$. Hence, 	
	\[
	d_{\bm{b},\bm{\mu}}(S_i,B_{S_i}+M_{S_i})\ge d_{\bm{b},\bm{\mu}}(S_{i+1},B_{S_{i+1}}+M_{S_{i+1}}).
	\] Moreover, if $S_i$ and $S_{i+1}$ are not isomorphic in codimension $1$, then the above inequality is strict. In fact, if there exists a divisor $E \subset S_i$ which is not on $S_{i+1}$, then $E$ is counted by the second summand in $d_{\bm{b},\bm{\mu}}(S_i,B_{S_i}+M_{S_i})$, while not counted in $d_{\bm{b},\bm{\mu}}(S_{i+1},B_{S_{i+1}}+M_{S_{i+1}})$. Similarly for the case that $E$ is on $S_{i+1}$ but not on $S_i$. For $i\gg 1$, we can assume that the difficulties are constant, and thus $S_i$ and $S_{i+1}$ are isomorphic in codimension $1$. This is the advantage of introducing the above difficulty: in \cite[Proposition 4.2.14]{Fujino07}, the case that $S_i \to T_i$ is a divisorial contraction but $S_{i+1} \to T_i$ is a small contraction cannot be excluded by the difficulty therein.

Let $T$ be the normalization of the image of $S_1$ (hence the image of any $S_i$) in $Z$, and $T_i$ be the normalization of the image of $S_i$ in $Z_i$. In general, $S_i\dashrightarrow S_{i+1}/T_i$ may not be a $(K_{S_i}+B_{S_i}+M_{S_i})$-flip$/T$. However, we can use the same argument as Subsection \ref{subsection: Lifting a sequence of flips with quasi-scaling} to construct a sequence of g-MMP$/T$ with scaling of an NQC divisor over some g-dlt modifications, $S_i' \dashrightarrow S_{i+1}'$. For simplicity, we just sketch the argument below.

Because $K_{X_1}+B_1+M_1+\nu_1 Q_1 \equiv 0/Z_1$, we have
\[
K_{S_1}+B_{S_1}+M_{S_1}+\nu_1 Q_{S_1} \equiv 0/T_1,
\] where $Q_{S_1} = Q_1|_{S_1}$ is defined inductively as follows (cf. \eqref{eq: E', P'}\eqref{eq: Q'}). Suppose that $\pi: \tilde X_1 \to X_1$ is a model of $X_1$ such that $P_1=\pi_*\tilde P_1$ with $\tilde P_1$ an NQC divisor. Then we have $\pi^*P_1 = \tilde P_1 +F$ with $F \geq 0$. Notice that $S_1$ is an irreducible component of $V_1 \cap \cdots \cap V_{n-k}$, where $V_i \subset \lf B_i \rf$ (see Proposition \ref{prop: intersection on g-lc centers}). We first define $Q_{V_1}$. Let $\pi$ also denote the induced morphism $\tilde V_1 \to V_1$. Then $\pi^*(P_1|_{V_1}) = \tilde P_1|_{\tilde V_1}+F|_{\tilde V_1}$. Here $\tilde P_1|_{\tilde  V_1}$ is still an NQC divisor, and $F|_{\tilde V_1}$ is an effective divisor. Because $\nu_1>0$, no component of $E_1$ is contained in $\lf B_1 \rf$, and thus $E_1|_{V_1} \geq 0$. Now set
\[
E_{V_1} = E_1|_{V_1} + \pi_*(F|_{\tilde V_1}) \text{~and~} P_{V_1} = \pi_*(\tilde P_1|_{\tilde V_1}),
\] and
\[
Q_{V_1} \coloneqq E_{V_1} + P_{V_1}= \pi_*(\pi^*(E_1|_{V_1}+P_1|_{V_1}) = Q_1|_{V_1}.\] We can repeat the above process to define $Q_{S_1}, P_{S_1}$. Notice that $P_{S_1}$ is a pushforward of an NQC divisor.

Let $K_{S'_1}+ B_{S'_1}+M_{S'_1} = h_i^*(K_{S_1}+ B_{S_1}+M_{S_1})$ be a $\Qq$-factorial g-dlt modification of $(S_1, B_{S_1}+M_{S_1})$. By the same argument as Subsection \ref{subsection: Lifting a sequence of flips with quasi-scaling}, we can run a g-MMP$/T_1$ on $K_{S'_1}+ B_{S'_1}+M_{S'_1}$ with scaling of $Q_1'$. It terminates with $(S'_2,  B_{S'_2}+M_{S'_2})$ which is a g-dlt modification of $(S_2,  B_{S_2}+M_{S_2})$. We can continue this process on $K_{S'_2}+B_{S'_2}+M_{S'_2}$. This gives a sequence of g-MMP$/T$ with scaling of $Q_{S_1'}$. If this sequence does not terminate. Then by the assumption, there exists a g-log minimal model$/T$ for $K_{S'_1}+B_{S'_1}+M_{S'_1}+\mu Q_{S'_1}$. By Theorem \ref{thm: gmmtermination}, the g-MMP terminates, and this is a contradiction. Hence the g-MMP$/T$ terminates, that is, $(S'_i, B_{S_i'}+M_{S_i'})= (S'_{i+1}, B_{S_{i+1}'}+M_{S_{i+1}'})$ for $i\gg 1$. This implies that $S_i \simeq S_{i+1}$ by \cite[Lemma 4.2.16]{Fujino07}.
\end{proof}

\section{Proofs of the main results}\label{sec: proof}

A birational NQC weak Zariski decomposition can be obtained from a g-log minimal model.

\begin{proposition}\label{prop: g-log minimal model implies NQC zarski decomposition}
	Let $(X/Z,B+M)$ be an NQC g-lc pair. Suppose that $(X/Z,B+M)$ has a g-log minimal model, then $(X/Z,B+M)$ admits a birational NQC weak Zariski decomposition. 
\end{proposition}
\begin{proof}
	Let $(Y/Z,B_Y+M_Y)$ be a g-log minimal model of $(X/Z,B+M)$. By Proposition \ref{le: decomposition to nef Cartier divisors}, there exist $\Qq$-Cartier nef$/Z$ divisors $M_i$, and $\mu_i\in\Rr_{>0}$, such that 
	\[
	K_{Y}+B_Y+M_Y=\sum \mu_i M_i.
	\]
	Let $p:W\to X,q:W\to Y$ be a common resolution of $X\dashrightarrow Y$, then
	\begin{align*}
	p^{*}(K_X+B+M)&=q^{*}(K_Y+B_Y+M_Y)+E\\
	&=\sum \mu_i q^{*}(M_i)+E,
	\end{align*} with $E \geq 0$. This is a birational NQC weak Zariski decomposition of $(X/Z,B+M)$.
\end{proof}

This shows one direction of Theorem \ref{thm: weak zariski equiv mm}. For the other direction, we first show the existence of g-log minimal models instead of g-log terminal models (see Definition \ref{def: g-log minimal model and g-log terminal model}).

\begin{definition}[\cite{Birkarweak12} Definition 2.1]\label{def: theta}
	For a g-pair $(X/Z, B+M)$ with boundary part $B$ and nef part $M$. Let $f:W\to X$ be a projective birational morphism from a normal variety, and $N$ be an effective $\Rr$-divisor on $W$. Let $f_{*}N=\sum_{i\in I} a_i N_i$ be a prime decompostion. We define
	\[
	\theta(X/Z,B+M,N) \coloneqq \#\{i\in I \mid N_i \text{~is not a component of~} \lf B\rf\}.
	\]
\end{definition}

\begin{definition}
A g-pair $(X/Z, B+M)$ is called \emph{log smooth} if $X$ is smooth, with data $X \xrightarrow{{\rm id}} X \to Z$ and $M$ (in particular, $M$ is nef$/Z$), and 
\[
\Supp(B)\bigcup\Supp(M)
\] is a simple normal crossing divisor. 
\end{definition}

\begin{theorem}\label{thm: weak zariski implies mm}
	Suppose that Conjecture \ref{conj: NQCbirweakzar} holds in dimensions $\le d$. Then g-log minimal models exist for pseudo-effective NQC g-lc pairs of dimensions $\le d$.
\end{theorem}

\begin{proof}

Step 1. It is enough to show Theorem \ref{thm: weak zariski implies mm} in the log smooth case (cf. \cite[Remark 2.6]{Birkar10} or \cite[Remark 2.8]{Birkar12}). In fact, let $(X/Z,B+M)$ be an NQC g-lc pair. Let $\pi:(W/Z,B_W+M_W)\to X$ be a log resolution of $(X,B+M)$, where $B_W$ is defined as \eqref{eq: B_Y}, and $M_W$ is an NQC divisor. Thus
\[
K_W+B_W+M_W=\pi^{*}(K_X+B+M)+F,
\] with $F\ge0$ an exceptional divisor. Let $(Y/Z,B_Y+M_Y)$ be a g-log minimal model of $(W/Z,B_W+M_W)$ and $D$ be a prime divisor on $X$ which is contracted over $Y$. Then,
	\[
	a(D,X,B+M)=a(D,W,B_W+M_W)<a(D,Y,B_Y+M_Y).
	\]
	This implies that $(Y/Z,B_Y+M_Y)$ is also a g-log minimal model of $(X/Z,B+M)$ (see Definition \ref{def: g-log minimal model and g-log terminal model}). 
	
	Assume that $\pi: W \to X$ is a sufficiently high model such that  $\pi^{*}(K_X+B+M)=N+P/Z$ is an NQC weak Zariski decomposition, where $P$ is an NQC divisor and $N$ is effective. Then 
	\[
	K_W+B_W+M_W=\pi^{*}(K_X+B+M)+F=(N+F)+P/Z.
	\]
	This is an NQC weak Zariski decomposition$/Z$ for $K_W+B_W+M_W$. Moreover, 
	\[
	\theta(X/Z,B+M,N)=\theta(W/Z,B_W+M_W,N)=\theta(W/Z,B_W+M_W,N+F).
	\]

Thus we may assume that $(X,B+M)$ is log smooth with $M$ an NQC divisor, and $K_X+B+M \equiv P+N/Z$ is an NQC weak Zariski decomposition. Moreover, by induction on dimensions, we can assume that Theorem \ref{thm: weak zariski implies mm} holds in dimensions $\leq d-1$. 

In the following, we prove Theorem \ref{thm: weak zariski implies mm} by induction on $\theta(X,B+M,N)$.

Step 2. When $\theta(X,B+M,N)=0$, we show Theorem \ref{thm: weak zariski implies mm}.

	By definition, $\theta(X,B+M,N)=0$ implies that $\Supp\lf B \rf \supset \Supp N$. By Lemma \ref{lem:glcklt}, for any $\beta_0 > 0$, we can run a g-MMP$/Z$ on $(K_X+B+M+\beta_0 P)$ with scaling of an ample divisor. By proposition \ref{prop: P trivial}, for $\beta_0\gg 1$, such g-MMP$/Z$ is $P$-trivial. Thus it is also a g-MMP$/Z$ on $(K_X+B+M)$. Moreover, by $K_X+B+M \equiv P+N/Z$ and $P$ is nef$/Z$, the contracting locus belongs to the birational transform of $\Supp N \subset \Supp B$. Because $M+\beta_0 P$ is NQC, the above g-MMP$/Z$ terminates by Theorem \ref{prop: special termination 2}. Let $(X_1, B_0+M_1+\beta P_1)$ be the corresponding g-log minimal model with $K_{X_1}+B_1+M_1 \equiv N_1+P_1/Z$. Moreover, $P_1$ is nef$/Z$.
	
	Next, we run a special kind of g-MMP$/Z$ on $(K_{X_1}+B_1+M_1)$ with scaling of $P_1$ as follows. 
	
	Suppose that we have constructed $(X_i,B_i+M_i)$. Let 
	\[
	\nu_i = \inf\{t \geq 1 \mid K_{X_i}+B_i+M_i+t P_i \text{~is nef~}/Z\}.
	\] 
	
	(i) If $\nu_i=0$. Then $({X_i}, B_i+M_i)$ is a g-log minimal model, and we have done. 
	
	(ii) If $0<\nu_i<\nu_{i-1}$ (we set $\nu_0=+\infty$). By Lemma \ref{lem:MMPscalingNQC}, there exists an extremal ray $R$ such that $(K_{X_i}+B_i+M_i)\cdot R<0$ and $(K_{X_i}+B_i+M_i+\nu_iP_i)\cdot R=0$. We contract $R$ and get a divisorial contraction or a flipping contraction. Let $X_i\dashrightarrow X_{i+1}$ be the corresponding divisorial contraction or flip.
	
	(iii) If $\nu_i=\nu_{i-1}>0$. Choose $\beta_i<\nu_i$ sufficiently close to $\nu_i$ ($\beta_i$ can be determined from the discussion later), we run a g-MMP$/Z$ with scaling of an ample$/Z$ divisor $H$ on $(K_{X_i}+B_i+M_i+\beta_i P_i)$. We claim that this g-MMP$/Z$ is also a g-MMP$/Z$ with scaling of $P_i$, and it terminates. Let $(X_{i+1}, B_{i+1}+M_{i+1})$ be the resulting g-pair. In particular, $\nu_{i+1} \leq \beta_i<\nu_i$.
	
\begin{proof}[Proof of the Claim in (iii)]	
	Since for any $\nu_i>\beta_i>0$, we have
	\begin{equation}\label{eq: K+B+M+betaP}
	\begin{split}
	&\frac{\nu_i}{\nu_i-\beta_{i}}(K_{X_i}+B_i+M_i+\beta_i P_i)\\
	=&(K_{X_i}+B_i+M_i)+\frac{\beta_i}{\nu_i-\beta_i}(K_{X_i}+B_i+M_i+\nu_i P_i).
	\end{split}
	\end{equation} 
	If $\beta_i$ is sufficiently close to $\nu_i$, then $\frac{\beta_i}{\nu_i-\beta_i}$ is sufficiently large. For a g-MMP$/Z$ with scaling of an ample$/Z$ divisor on $(K_{X_i}+B_i+M_i+\beta_i P_i)$, by \eqref{eq: K+B+M+betaP}, it is also a g-MMP$/Z$ on $(K_{X_i}+B_i+M_i)+\frac{\beta_i}{\nu_i-\beta_i}(K_{X_i}+B_i+M_i+\nu_i P_i)$. By Proposition \ref{le: decomposition to nef Cartier divisors}, $K_{X_i}+B_i+M_i+\nu_i P_i$ is an NQC divisor$/Z$. Hence by Proposition \ref{prop: P trivial},  for a sufficiently large $\frac{\beta_i}{\nu_i-\beta_i}$, this g-MMP$/Z$ is $(K_{X_i}+B_i+M_i+\nu_i P_i)$-trivial. By $\beta_i<\nu_i$, it is also a g-MMP$/Z$ on $(K_{X_i}+B_i+M_i)$ with scaling of $P_i$. In fact, if $Y$ is an intermediate variety in this g-MMP$/Z$, for a contracting curve $\Gamma$,
	\[
	(K_{Y}+B_Y+M_Y+\beta_i P_Y)\cdot \Gamma<0 \text{~and~} (K_{X_Y}+B_Y+M_Y+\nu_i P_Y) \cdot \Gamma=0,
	\] thus $(K_{X_Y}+B_Y+M_Y)\cdot \Gamma<0$ and $P_Y \cdot \Gamma>0$.
	
	Next, we show that the g-MMP$/Z$ terminates. Because
	\begin{align*}
	K_{X_i}+B_i+M_i&=N_i+P_i\\
	&=\frac{1}{1+\nu_i}(N_i+(1+\nu_i)P_i)+\frac{\nu_i}{1+\nu_i}N_i\\
	&=\frac{1}{1+\nu_i}(K_{X_i}+B_i+M_i+\nu_i P_i)+\frac{\nu_i}{1+\nu_i}N_i,
	\end{align*}
	a flipping curve intersects the birational transform of $N_i$ negatively. Thus the flipping locus is contained in the birational transform of $\Supp N_i \subseteq \Supp \lfloor B_i\rfloor$. Suppose that the g-MMP$/Z$ does not terminate. Because the g-MMP$/Z$ is a scaling of an ample$/Z$ divisor, the corresponding nef thresholds $\nu_j'$ satisfies $\nu_j' \neq \lim \nu_j'$. Otherwise, $\mu'= \lim \nu_j'>0$, then the g-MMP$/Z$ can be viewed as a g-MMP$/Z$ on $K_{X_i}+B_i+M_i+\beta_iP_i+\mu' H$, then it terminates by Lemma \ref{lem:glcklt} and \cite[Corollary 1.4.2]{BCHM10}. However, by Theorem \ref{prop: special termination 2}, the above g-MMP$/Z$ terminates. This is a contradiction.	
\end{proof}
	
By applying (i)-(iii), we obtain a g-MMP$/Z$ on $K_{X_i}+B_i+M_i$ with scaling of $P_1$, 
\[
\dashrightarrow X_i = Y_i^{1} \dashrightarrow \cdots \dashrightarrow Y_i^{k_1} = X_{i+1} \dashrightarrow.
\] Let $\tilde\nu_j$ be the corresponding nef thresholds. Then either the g-MMP$/Z$ terminates or 
\[
\lim\tilde \nu_j = \lim\nu_i>\tilde \nu_j.
\] Moreover, as $K_{Y^j_i}+B_{Y^j_i}+M_{Y^j_i}=N_{Y^j_i}+P_{Y^j_i}$ and $P_{Y^j_i} \cdot \Gamma>0$ for a contracting curve $\Gamma$, we have $N_{Y^j_i} \cdot \Gamma<0$. Thus the flipping locus is contained in the birational transform of $\Supp N_i \subseteq \Supp \lfloor B_i\rfloor$. By Theorem \ref{prop: special termination 2} again, the g-MMP$/Z$ terminates. This finishes the proof of the $\theta(X,B+M,N)=0$ case.

Step 3. Next we show the induction step. The argument is identical to \cite[Proof of Theorem 1.5]{Birkarweak12}, except that we deal with the g-pairs. 

First, for a divisor $D=\sum_i d_i D_i$, we write $D^{\leq 1} \coloneqq \sum_i \min\{d_i, 1\}D_i$. Suppose that Theorem \ref{thm: weak zariski implies mm} does not hold. We assume that $\theta(X,B+M,N) \geq 1$ is the minimal number such that $K_X+B+M$ does not have a log minimal model. By Step 1, we can assume that $(X,B+M)$ is log smooth. By $\theta(X,B+M,N) \geq 1$,
	\[
	\alpha \coloneqq \min\{t>0 \mid \lf (B+tN)^{\leq 1} \rf \neq \lf B \rf\} 
	\] is a finite number. Let $C$ be the divisor such that $(B+\alpha N)^{\leq 1} = B+C$. Thus $\Supp C \subseteq \Supp N$, and 
	\begin{equation}\label{eq: components of C}
	\theta(X,B+M,N) = \#\{\text{components of~}C\}.
	\end{equation} Let $A \geq 0$ satisfy $\alpha N = C+A$, then $\Supp A \subseteq \Supp \lf B \rf$, and $A=(B+\alpha N)-(B+\alpha N)^{\leq 1}$.
	
	Because $\theta(X, (B+C)+M,N+C)<\theta(X, B+M,N)$, by the induction hypothesis, $(X, (B+C)+M)$ has a log minimal model, $(Y,(B+C)_Y+M_Y)$. Notice that $(X, (B+C)+M)$ is a g-lc pair with boundary part $B+C$ and nef part $M$. Let $g: U \to X, h: U \to Y$ be a sufficiently high log resolution, then
	\begin{equation}\label{eq: common U}
	g^*(K_X+(B+C)+M)=h^*(K_Y+ (B+C)_Y+M_Y)+N',
	\end{equation} with $N' \geq 0$ and $h$-exceptional.
	Let 
	\[
	P' \coloneqq  h^*(K_Y+ (B+C)_Y+M_Y),
	\] then it is nef$/Z$ and NQC by Proposition \ref{le: decomposition to nef Cartier divisors}. Thus $P'+N'$ is an NQC weak Zariski decomposition$/Z$ for $g^*(K_X+(B+C)+M)$. We have
	\[
	g^{*}(N+C)-N'=h^{*}(K_Y+ (B+C)_Y+M_Y)-g^{*}P.
	\]
	Since $h^{*}(K_Y+ (B+C)_Y+M_Y)-g^{*}P$ is anti-nef$/Y$, by the negativity lemma, $N' \leq g^*(N+C)$. As $\Supp C \subseteq \Supp N$, we have $\Supp N'\subseteq \Supp g^{*}N$. 
	
	By above, we have
	\[
	\begin{split}
	(1+\alpha)g^*(K_X+B+M) &= g^*(K_X+B+M)+\alpha g^*P+\alpha g^*N\\
	&=g^*(K_X+B+M)+\alpha g^*P+ g^*C+ g^*A\\
	& = P'+N'+\alpha g^*P+ g^*A.
	\end{split}
	\] 
	Thus,
	\[g^*(K_X+B+M) = \frac{1}{1+\alpha}(P'+\alpha g^*P)+\frac{1}{1+\alpha}(N'+g^*A).\]
	Set 
	\[
	P'' \coloneqq  \frac{1}{1+\alpha}(P'+\alpha g^*P) \text{~and~} N''\coloneqq \frac{1}{1+\alpha}(N'+g^*A),
	\] then
	\[
	g^*(K_X+B+M)=P''+N''
	\]
	is a birational NQC weak Zariski decomposition$/Z$ for $K_X+B+M$. Since $\alpha N=C+A$, we have $\Supp N'' \subseteq \Supp g^*N$, and thus $\Supp g_* N'' \subseteq \Supp N$. Because $\theta(X, B+M,N)$ is minimal, $\theta(X, B+M,N) = \theta(X, B+M,g_{*}N'')$, and every component of $C$ is also a component of $g_*N''$ according to \eqref{eq: components of C}. Because $A, C$ do not have common components, we have $\Supp C \subseteq \Supp g_*N'$, and thus $C$ is exceptional over $Y$ by \eqref{eq: common U}. Hence by definition $(B+C)_Y=B_Y$, and $P' = h^*(K_Y+B_Y+M_Y)$. We will compare the g-log discrepancies below.
	
	Let $G \geq 0$ be the largest divisor such that $G \leq g^{*}C$ and $G \leq N'$. Set $\tilde C = g^{*}C-G, \tilde N' = N' -G$. By \eqref{eq: common U}, we have 
	\begin{equation}\label{eq: tilde C}
	g^*(K_X+B+M)+\tilde C = P'+\tilde N'.
	\end{equation}
	
(i) If $\tilde C$ is exceptional over $X$, then because $g^*(K_X+B+M)-P'=\tilde N'-\tilde C$ is anti-nef over $X$, by the negativity lemma, $\tilde N'-\tilde C \geq 0$, which implies $\tilde C=0$ as $\tilde C$ and $\tilde N'$ have no common components. From \eqref{eq: tilde C},
\[
g^*(K_X+B+M)= P'+\tilde N'
\]
	is a birational NQC-weak Zariski decomposition$/Z$ for $K_X+B+M$. Moreover,
	\[
	\begin{split}
	&g^*(K_X+B+M)-h^*(K_Y+B_Y+M_Y)\\
	=&\sum_D \left(a(D; Y, B_Y+M_Y)-a(D;X, B+M)\right)D\\
	=&\tilde N',
	\end{split}
	\] 
	where $D$ runs over the prime divisors on $U$.
	
	(ia) Suppose that $\Supp g_*\tilde N' = \Supp g_*N'$. Then by \eqref{eq: common U}, $\Supp \tilde N'$ contains the birational transform of all the prime exceptional$/Y$ divisors on $X$. Hence $(Y, B_Y+M_Y)$ is also a g-log minimal model of $(X, B+M)$, a contradiction.

	(ib) Hence we can assume $\Supp g_*\tilde N' \subsetneq \Supp g_*N'$. Thus,
	$$\Supp(g_{*}N'-g_{*}G)=\Supp g_{*}\tilde N'\subsetneq \Supp g_{*} N'\subseteq \Supp N.$$
	Since $G$ is the largest divisor such that $G \leq g^{*}C$ and $G \leq N'$, some component of $C$ is not a component of $g_*\tilde N'$. By \eqref{eq: components of C}, we have
	\[
	\theta(X/Z, B+M, \tilde N') <\theta(X/Z, B+M, N),
	\] 
	which contradicts the minimality of $\theta(X/Z, B+M, N)$.

	(ii) Hence $\tilde C$ is not exceptional over $X$. Let $\beta>0$ be the smallest number such that
	\[
	\tilde A \coloneqq \beta g^*N-\tilde C \text{~and~} g_*\tilde A \geq 0.
	\] Then there exists a component $D$ of $g_*\tilde C$ which is not a component of $g_*\tilde A$. We have
	\[
	\begin{split}
	(1+\beta)g^*(K_X+B+M) &= g^*(K_X+B+M) +\tilde C+\tilde A+\beta g^*P\\
	&=P'+\beta g^*P+\tilde N' + \tilde A.
	\end{split}
	\] 
	By the negativity lemma, we have $\tilde N' + \tilde A \geq 0$. Let 
	\[
	P''' \coloneqq\frac{1}{1+\beta}(P'+\beta g^*P) \text{~and~}  N'''\coloneqq\frac{1}{1+\beta}(\tilde N' + \tilde A),
	\] then
	\[
	g^*(K_X+B+M) = P'''+N'''
	\] is a birational NQC weak Zariski decomposition of $K_X+B+M$. By construction, $D$ is a component of $\Supp g_{*}\tilde C\subseteq \Supp C \subseteq \Supp N$. Since $\Supp\tilde C \cap \Supp\tilde N' = \emptyset$, $D$ is not a component of $g_{*}\tilde N'$. Thus, $D$ is not a component of 
	\[
	\Supp( g_*N''')=\Supp(g_{*}\tilde N')\cup\Supp (g_{*}\tilde {A}).
	\] Hence
	\[
	\theta(X, B+M,N)>\theta(X,B+M,N'''),
	\]
	which still contradicts the minimality of $\theta(X/Z, B+M, N)$.
\end{proof}

\begin{proof}[Proof of Theorem \ref{thm: weak zariski scaling ample}]
	Let  $\nu_i$ be the nef threshold in the g-MMP$/Z$ with scaling of an ample$/Z$ divisor $A$ (see Definition \ref{defn:MMPsP}). By Lemma \ref{lem:glcklt}, for any $\epsilon>0$, there exists a klt pair $(X,\Delta_{\epsilon})$, such that $\Delta_{\epsilon}\sim_{\Rr} B+M+\epsilon A/Z$. If $\lim \nu_i = \mu>0$, then this g-MMP$/Z$ is also a MMP$/Z$ on $K_X+\Delta_{\mu} \sim_\Rr B+M+\mu A/Z$. By \cite[Corollary 1.4.2]{BCHM10}, it terminates. Hence we have $\mu=0$. If this g-MMP$/Z$ does not terminate, we have $\nu_i>\mu=0$ for all $i$. By Theorem \ref{thm: weak zariski implies mm}, $(X/Z,B+M)$ has a g-log minimal model $(Y/Z,B_Y+M_Y)$. Hence the g-MMP terminates by Theorem \ref{thm: gmmtermination}.
\end{proof}

\begin{proof}[Proof of Theorem \ref{thm: weak zariski equiv mm}]
The g-minimal model conjecture (Conjecture \ref{conj:glogmmp}) implies the birational weak Zariski decomposition conjecture (Conjecture \ref{conj: NQCbirweakzar}) by Proposition \ref{prop: g-log minimal model implies NQC zarski decomposition}. For the other direction, Theorem \ref{thm: weak zariski scaling ample} implies that any g-MMP$/Z$ with scaling of an ample divisor terminates. The resulting model is a g-log terminal model as a g-MMP$/Z$ does not extract divisors.
\end{proof}

\begin{proof}[Proof of Theorem \ref{thm: ter lc implies g-log minimal model}] First, we show that $(X/Z,B+M)$ has a g-log minimal model (see Definition \ref{def: g-log minimal model and g-log terminal model}). By Step 1 of the proof of Theorem \ref{thm: weak zariski implies mm}, we can assume that $(X/Z,B+M)$ is log smooth. By Lemma \ref{prop: P trivial}, there exists a $\beta \gg1$, such that a g-MMP$/Z$ with scaling of an ample divisor on $(K_X+B+M+\beta M)$ is $M$-trivial. Thus, this g-MMP$/Z$ is also a MMP$/Z$ on the dlt pair $(K_X+B)$. By assumption, it terminates with a model $Y$. Since $(X/Z,B+M)$ is pseudo-effective, $X\dashrightarrow Y/Z$ is birational and $K_Y+B_Y+(\beta+1) M_Y$ is nef. By Lemma \ref{lem:MMPscalingNQC}, we can run a g-MMP$/Z$ on $(K_Y+B_Y+M_Y)$ with scaling of $M_Y$. This g-MMP is also a MMP$/Z$ on $(K_Y+B_Y)$. By assumption, it terminates with $Y'$.      Because the sequence $X \dashrightarrow Y \dashrightarrow Y'$ is also a g-MMP$/Z$ on $K_X+B+M$, $(Y'/Z,B_{Y'}+M_{Y'})$ is a desired g-log minimal model.
	
From a g-log minimal model to a g-log terminal model, we use the same argument as Theorem \ref{thm: weak zariski equiv mm}. In fact, the existence of g-log minimal model implies the existence of birational weak Zariski decomposition by Proposition \ref{prop: g-log minimal model implies NQC zarski decomposition}, then by Theorem \ref{thm: weak zariski scaling ample}, any g-MMP$/Z$ with scaling of an ample divisor terminates. The resulting model is a g-log terminal model.
\end{proof}

\bibliographystyle{abbrv}

\end{document}